\documentclass{amsart}
\usepackage{graphicx} 

\usepackage[title]{appendix}
\usepackage[utf8]{inputenc}
\usepackage[utf8]{inputenc}
\usepackage{enumitem}
\usepackage{enumitem}
\usepackage{enumitem}
\usepackage{graphicx}
\usepackage{graphics}
\usepackage{amssymb} 
\usepackage[T1]{fontenc} 
\usepackage{amsmath,amsthm,amssymb}
\usepackage{verbatim}
\usepackage [english]{babel}
\usepackage{graphicx}

\usepackage [autostyle, english = american]{csquotes}
\MakeOuterQuote{"}
\input xy
\xyoption{all}

\providecommand{\MR}{\relax\ifhmode\unskip\space\fi MR }

\providecommand{\href}[2]{#2}

\newenvironment{altenumerate}
   {\begin{list}
      {\textup{(\theenumi)} }
      {\usecounter{enumi}
       \setlength{\labelwidth}{0pt}
       \setlength{\labelsep}{2pt}
       \setlength{\leftmargin}{0pt}
       \setlength{\itemsep}{\the\smallskipamount}
       \renewcommand{\theenumi}{\roman{enumi}}
      }}
   {\end{list}}

\usepackage{float}
\usepackage{framed}
\usepackage{hyperref}
\usepackage{graphicx}

\usepackage{amsfonts}
\usepackage{amsthm}
\usepackage{amsmath}
\usepackage{amscd}
\usepackage{graphicx}
\usepackage{graphics}
\usepackage{pict2e}
\usepackage{epic}
\usepackage{centernot}
\usepackage{mathtools}

\usepackage[outline]{contour}
\usepackage{hyperref}
\usepackage[mathscr]{euscript}

\DeclareSymbolFont{rsfs}{U}{rsfs}{m}{n}
\DeclareSymbolFontAlphabet{\mathscrsfs}{rsfs}

\DeclareMathOperator{\ab}{ab}

\def\Q{\mathbb{Q}}

\def\Z{\mathbb{Z}}

\def\ad{\mathrm{ad}}

\def\int{\mathrm{int}}

\def\isom{\cong}
\def\Ga{\mathbb{G}_{\text{a}}}
\def\Gm{\mathbb{G}_{\text{m}}}

\DeclareMathOperator{\Aut}{Aut}

\DeclareMathOperator{\W}{W}

\usepackage{amsmath,bm}
\numberwithin{equation}{section}
\newtheorem{Theorem}{Theorem}[section]
\numberwithin{Theorem}{section}
\newtheorem{lemma}[Theorem]{Lemma}
\newtheorem{Cor}[Theorem]{Corollary}
\newtheorem{prop}[Theorem]{Proposition}

\newtheorem{defn}[Theorem]{Definition}

\theoremstyle{definition}
\newtheorem{rmk}[Theorem]{Remark}
\newtheorem{example}[Theorem]{Example}

\title{On the symmetric monoidal structure of the Geometric Whittaker Model}

\date{}
\newcommand{\beq}{\begin{equation}}
\newcommand{\eeq}{\end{equation}}
\newcommand{\bthm}{\begin {theorem}}
\newcommand{\ethm}{\end {theorem}}
\newcommand{\bprop}{\begin {proposition}}
\newcommand{\eprop}{\end {proposition}}
\newcommand{\bprob}{\begin {prob}}
\newcommand{\eprob}{\end {prob}}
\newcommand{\bcor}{\begin {corollary}}
\newcommand{\ecor}{\end {corollary}}
\newcommand{\blem}{\begin{lemma}}
\newcommand{\elem}{\end{lemma}}
\newcommand{\bdefn}{\begin{defn}}
\newcommand{\edefn}{\end{defn}}
\newcommand{\bconj}{\begin{conjecture}}
\newcommand{\econj}{\end{conjecture}}
\newcommand{\brk}{\begin{rk}}
\newcommand{\erk}{\end{rk}}
\newcommand{\bpf}{\begin{proof}}
\newcommand{\epf}{\end{proof}}
\newcommand{\bex}{\begin{ex}}
\newcommand{\eex}{\end{ex}}
\newcommand{\bit}{\begin{itemize}}
\newcommand{\eit}{\end{itemize}}

\renewcommand{\subset}{\subseteq}

\newcommand{\D}{\mathscrsfs{D}}

\newcommand{\Dcirc}{\D^\circ}

\newcommand{\into}{\hookrightarrow}

\newcommand{\opname}{\operatorname}
\newcommand{\Perv}{\opname{Perv}}

\newcommand{\id}{\operatorname{id}}
\newcommand{\Id}{\operatorname{Id}}

\newcommand{\Qlcl} {\overline{\mathbb{Q}}_{\ell}}

\renewcommand{\L}{\mathcal{L}}

\newcommand{\Hcal}{\mathcal{H}}

\newcommand{\eL}{e_\L}

\newcommand{\Uop}{{U^{-}}}

\begin{document}
\author{Ashutosh Roy Choudhury and Tanmay Deshpande}

\address{Boston University Department of Mathematics and Statistics, 665 Commonwealth Avenue, Boston,
MA, USA}
\email{ashutosh@bu.edu}
\address{Tata Institute of Fundamental Research, Homi Bhabha Road, Mumbai 400005, INDIA}
\email{tanmay@math.tifr.res.in}
\maketitle
\begin{abstract}
Let $G$ be a connected reductive algebraic group over an algebraically closed field $k$ of characteristic $p > 0$ and let $\ell$ be a prime number different from $p$. Let $U \subseteq G$ be a maximal unipotent subgroup, $T$ a maximal torus normalizing $U$ and $W$ the Weyl group of $G$. Let $\mathcal{L}$ be a non-degenerate multiplicative $\overline{\mathbb{Q}}_{\ell} $-local system
on $U$. In \cite{bd}, the authors show that the bi-Whittaker category, namely the triangulated monoidal category of $(U, \mathcal{L})$-biequivariant $\overline{\mathbb{Q}}_{\ell}$-complexes on $G$ is monoidally equivalent to an explicit thick triangulated monoidal subcategory $\mathscrsfs{D}_{W}^{\circ}(T) \subseteq \mathscrsfs{D}_{W}(T)$  of "central sheaves"
on the torus. In particular it has the structure of a symmetric
monoidal category coming from the symmetric monoidal structure on $\mathscrsfs{D}_W(T)$. 

In this paper, we give another construction of a symmetric monoidal structure on the bi-Whittaker category and prove that it agrees with the one coming from the above construction. For this, among other things, we geometrize a proof in the setting of finite groups of Lie type to the geometric setup.

\end{abstract}
\tableofcontents
\section{Introduction}
Let $G$ be a connected reductive algebraic group over an algebraically closed field $k$ of characteristic $p>0$.  Let us fix a prime number $\ell$ which is invertible in $k$. For a $k$-scheme $X$, let $\mathscrsfs{D}(X)$ denote the $\overline{\mathbb{Q}}_{\ell}$-linear triangulated category of bounded constructible $\overline{\mathbb{Q}}_{\ell}$-complexes on $X$. Following \cite{lo}, if $G$ acts on $X$, we have the equivariant derived category $\D_G(X) := \D(G{\backslash}X)$, the $\Qlcl$-linear triangulated category of $\Qlcl$-complexes on the quotient stack $G{\backslash}X$. Note that we have the forgetful functor $\D_G(X) \to \D(X)$ and by an abuse of notation, we will often use the same symbol to denote an object of $\D_G(X)$ and its image after applying the forgetful functor.

We will use the following notation related to the reductive group $G$. Let $B\subseteq G$ be a Borel subgroup and let $T\subseteq B$ be a maximal torus. Let $U$ be the unipotent radical of $B$ so that $B = TU$. Let $B^{-} = TU^{-}$ be the opposite Borel. Let $\Delta \subseteq \bm{\phi}^{+}\subseteq \bm{\phi} $ denote the set of simple roots, positive roots and the roots respectively, which are determined by the choice of the pair $T\subseteq B$. Let $N = N(T)$ denote the normalizer of the maximal torus $T$ and let $W = N(T)/T$ denote the Weyl group of $G$. We denote the lattice of characters
of $T$ by $X^{\star}(T)$ and the dual lattice of cocharacters of $T$ by $X_{\star}(T)$.

Note that the category $\D(G)$ gets the structure of a monoidal category through convolution with compact support. Let $\mathcal{L}$ be a non-degenerate multiplicative local system on the maximal unipotent $U$. One has the closed idempotents $e_{\mathcal{L}}:= \mathcal{L}[2\text{dim 
 }U](\text{dim }U)$ and $e_{U}:= {\Qlcl}_{|U}[2\text{dim 
 }U](\text{dim }U)$ (refer \cite{bdr}) in $\D(G)$. In \cite{bd}, the authors study the bi-Whittaker category of $G$, namely the monoidal category of $(U,\mathcal{L})$-biequivariant sheaves on $G$, with the monoidal structure being induced by the one on $\mathscrsfs{D}(G)$. It can be identified with the full subcategory $e_{\mathcal{L}}\mathscrsfs{D}(G)e_{\mathcal{L}} \subset \mathscrsfs{D}(G)$ which is itself monoidal with unit object $e_{\L}$. In loc. cit., the authors identify it with a certain full thick triangulated monoidal subcategory $\mathscrsfs{D}^{\circ}_W(T)$ (of $W$-equivariant central sheaves, refer \ref{central}) of the triangulated symmetric monoidal category $\mathscrsfs{D}_W(T)$:

\begin{Theorem} (\cite[Thm. 1.4]{bd})\label{zetathm}
There is a triangulated monoidal equivalence.
\[  \zeta: e_{\mathcal {L}} \mathscrsfs{D}(G) e_{\mathcal{L}} \xrightarrow[]{\cong}
 \mathscrsfs{D}_{W}^{\circ}(T) \]
 whose inverse is given by the composition 
 \[\mathscrsfs{D}^{\circ}_W(T) \xrightarrow{\text{ind}^W}\mathscrsfs{D}_G^{\circ}(G)\xrightarrow{\text{HC}_{\mathcal{L}}} e_{\mathcal {L}} \mathscrsfs{D}(G) e_{\mathcal{L}}\]

\end{Theorem}

 (For details on the functors $\text{ind}^{W}$ and $\text{HC}_{\mathcal{L}}$, refer \cite{bd}.)
 
 As a consequence of the above result, one sees that the bi-Whittaker category has the structure of a symmetric monoidal category. This can be thought of as a geometric analogue of the “uniqueness of the
Whittaker model” or the multiplicity-freeness of the Gelfand-Graev representations, which is proved by showing that the endomorphism algebra of the Gelfand-Graev representation is a commutative algebra. The bi-Whittaker monoidal category considered here is the geometric analogue of this endomorphism algebra.

In their proof of these results, the Yokonuma-Hecke category $\mathscrsfs{D}(U^{-} \backslash G/U^{-})\cong e_{U^{-}}\mathscrsfs{D}(G)e_{U^{-}}$ plays an important role. In particular certain perverse sheaves $\mathcal{K}_w \in \mathscrsfs{D}(U^{-}\backslash G/U^{-})$ indexed by $w \in W$ and defined by Kazhdan and Laumon in \cite{kl} are used to construct the $W$-equivariance structure mentioned above.

In this paper, we give a different construction of the symmetric monoidal structure via geometrizing Gelfand's proof of the uniqueness of the Whittaker model. We refer \cite{car} for the proof for finite groups of Lie type which uses Gelfand's trick to show commutativity of the endomorphism algebra of the Gelfand-Graev representation: namely, by constructing an anti-automorphism of this algebra and proving that the anti-automorphism is in fact equal to the identity. In our geometric set-up we construct a  canonical anti-involution $\bm{\Psi}:G\to G$, given the non-degenerate multiplicative local system $\mathcal{L}$ on $U$. This uses the observation that a choice of such an $\mathcal{L}$ determines a pinning on the reductive group, using which we can uniquely construct an anti-involution that will preserve the local system $\mathcal{L}$, along with a few other properties (refer Prop. \ref{psiprop}). 
Since $\bm{\Psi}$ is an anti-involution, the pullback functor $\bm{\Psi}^{\ast}:\D(G)\to \D(G)$ has a canonical anti-monoidal structure. Also since $\bm{\Psi}^{\ast}\L\cong \L$ by construction, the bi-Whittaker full subcategory $\eL\D(G)\eL\subset \D(G)$ is preserved by $\bm{\Psi}^*$. 

The key point will be to construct a natural isomorphism between the anti-monoidal functor $\bm{\Psi}^{\ast}$ and the (monoidal) identity functor $\text{Id}$ on the bi-Whittaker category $\eL\D(G)\eL\subset \D(G)$. Indeed, given such a natural isomorphism $\bm{\Psi}^*\cong \text{Id}$, we can construct a candidate for braiding isomorphisms as follows: For $\mathcal{F},\mathcal{G}\in e_\L\D(G)e_\L$ we obtain a sequence of natural isomorphisms as follows
\beq\label{eq:braidingcandidate}
\mathcal{F}\ast \mathcal{G}\cong \bm{\Psi}^*(\mathcal{F}\ast\mathcal{G})\cong \bm{\Psi}^*(\mathcal{G})\ast \bm{\Psi}^*(\mathcal{F})\cong \mathcal{G}\ast\mathcal{F}.
\eeq
We can now state our main result.
\begin{Theorem}\label{mainthm}(Theorem \ref{nt}, Theorem \ref{zetathm}) (i) The functor $\bm{\Psi}^{\ast}: 
e_{\mathcal{L}}\mathscrsfs{D}(G)e_{\mathcal{L}} \to e_{\mathcal{L}}\mathscrsfs{D}(G)e_{\mathcal{L}}$
admits a natural isomorphism $\theta: \bm{\Psi}^{\ast} \to \text{Id}$ to the identity functor. \\
(ii) The natural isomorphism $\theta$ equips $e_{\mathcal{L}}\mathscrsfs{D}(G)e_{\mathcal{L}}$  with a symmetric monoidal structure defined using Equation \ref{eq:braidingcandidate} above.\\ 
(iii) The above symmetric monoidal structure agrees with the one coming from the equivalence in Theorem \ref{zetathm} from \cite{bd}.
\end{Theorem}

There are two main difficulties that arise in the geometric and categorical set-up compared to the set-up of finite groups of Lie type. Firstly, in the geometric set-up, it is not sufficient to construct the natural isomorphism $\theta: \bm{\Psi}^{\ast} \to \text{Id}$ on the restriction to the double cosets $BwB\subset G$.

The construction of the natural isomorphism will essentially use the Bruhat decomposition of the reductive group $G$, along with some facts about perverse sheaves and their gluing set up by Beilinson, only slightly adapted to our $(U,\L)$-biequivariant derived category. For this, we need to construct suitable functions on the closures of Bruhat cells to make Beilinson's set up work: 

\begin{prop} (Corollary \ref{functions}) 
 Fix a $w \in W$. There exist regular functions $f_{w}$ on $\overline{BwB}$, such that they vanish with non-zero multiplicity on each Bruhat cell $BvB$ that appears in $\overline{BwB} \backslash BwB = \bigcup_{v < w} BvB$ and are non-vanishing on the open cell $BwB$.   
\end{prop}

In this regard, this construction is an instance of the gluing of equivariant perverse sheaves (for the action of unipotent groups). This is in the setup developed by Laszlo and Olsson in their paper on the six-functor formalism on Artin Stacks (\cite{lo}).

The second difficulty in the categorical set-up is as follows: The natural transformation $\theta$ equips the identity functor on $e_{\L}\D(G)e_{\L}$ with an anti-monoidal structure. This gives a candidate for a braiding defined by Equation \ref{eq:braidingcandidate} satisfying the Yang-Baxter relation. In general such a structure need not necessarily satisfy the axioms of a braided monoidal category, namely the hexagonal axioms. To show that we indeed have a braided, and in fact, symmetric monoidal structure, it remains to verify the axioms and to reconcile it with the symmetric monoidal structure coming from the results in \cite{bd}. 

We will in fact prove that the candidate braiding that we construct agrees with the symmetric monoidal braiding that is obtained using the results of \cite{bd}. For this, we will prove the following in Section \ref{symmeq}:  
There is a natural isomorphism making the following diagram commute: 

\[\xymatrix{
 e_\mathcal{L}\mathscrsfs{D}(G) e_{\mathcal{L}}\ar[r]^{\bm{\Psi}^{\ast}} \ar[d]^{\zeta} & e_{\mathcal{L}}\mathscrsfs{D}(G)e_{\mathcal{L}} \ar[d]^{\zeta}\\
 \mathscrsfs{D}_W^{\circ}(T) \ar[r]^{s_{w_0}} & \mathscrsfs{D}_W^{\circ}(T)  .
}\] 
And secondly, the natural transformation $\theta: \bm{\Psi}^{\ast} \to \text{Id}$ goes to the natural transformation $\alpha_{w_0}: s_{w_0}\to \text{Id}$, under the equivalence $\zeta$ (with $\alpha_{w_0}$, $s_{w_0}$ as defined in \ref{natmaps}, roughly, $s_{w_0}$ translates the $W$-equivariant structure by $w_0$ and $\alpha_{w_0}$ is the obvious way of relating the new $W$-equivariant structure with the old one). We also need the following slightly technical result to complete the proof Theorem \ref{mainthm}: 
\begin{Theorem} (Theorem \ref{endid})
All the natural transformations between the identity functors on $e_{\mathcal{L}}\mathscrsfs{D}(G)e_{\mathcal{L}}$ which are induced from the Perverse subcategory $\text{Perv}(e_{\mathcal{L}}\mathscrsfs{D}(G)e_{\mathcal{L}})$ are given by multiplication by a scalar $c\in \overline{\Q_{\ell}}$ and hence the same also holds true for $\mathscrsfs{D}_W^{\circ}(T)$.     
\end{Theorem}
 For this, we have to use some results about the Fourier-Mellin transform on a torus (refer \ref{mellindef}) established by Gabber and Loeser in \cite{GL}.   

\section*{Acknowledgements}

We would like to thank Prof. Najmuddin Fakhruddin, Arnab Roy for useful discussions and Prof. Michel Brion for very helpful exchanges over email. This work was supported by the Department of Atomic Energy, Government
of India, under project no.12-R\&D-TFR-5.01-0500.

\section{Linear Sums of Schubert Varieties as Cartier Divisors}\label{Schubertsection}

In this section, we want to setup the necessary ingredients to perform Beilinson's construction so as to glue perverse sheaves across the strata of the Bruhat decomposition. 

We consider the problem in terms of Schubert varieties, which are $X_w :=  \overline{BwB}/B \subseteq G/B$, and linear sums of Schubert subvarieties that are Cartier divisors. Let $C_w:= BwB/B\subset X_w$. We want to consider line bundles on $X_w$ that pull back to the trivial bundle on $\overline{BwB}$. If we can find sections of line bundles on $X_w$ which vanish only on Schubert subvarieties $X_v$ which are of codimension $1$, then pulling back such sections to $\overline{BwB}$ will give us the kind of functions we want to run a gluing argument. We refer to Michel Brion's lectures (\cite{mb}) for more on Schubert varieties. \\
We first recall the definition of homogeneous line bundles on $X = G/B$ :
\begin{defn}
 Let $\lambda:B \to \Gm$ be a character of $B$. Let $B$ act on the product $G \times \mathbb{A}^1$ by $b(g, t) := (gb^{-1} , \lambda(b)t).$
This action is free, and the quotient
\[L_{\lambda} = G \times^{B} \mathbb{A}^1 := (G \times \mathbb{A}^1)/B\]
maps to $G/B$ via $(g,t)B \to gB$. This makes $L_{\lambda}$ the total space of a line bundle over $G/B$, the {\em homogeneous line bundle associated to the weight} $\lambda$.
\end{defn}

Note that $G$ acts on $L_{\lambda}$ via $g(h, t)B := (gh, t)B$, and that the projection $f : L_{\lambda} \to G/B$
is $G$-equivariant; further, any $g \in G$ induces a linear map from the fiber $f^{-1} (x)$  to $f^{-1}(gx)$. In other words, $L_{\lambda}$ is a $G$-linearized line bundle on $X$.

Recall that the character $\lambda$ must factor as $\lambda:B\to T\to \Gm$ since $B = T U$, and $U$ has no non-trivial characters. So we can think of $\lambda$ as a character of $T$, i.e. $\lambda\in X^\star(T)$.

Schubert varieties are normal. An argument can be found in \cite{mb} Section 2, originally due to Seshadri, \cite{cs}. (For a short proof using Frobenius Splitting of Schubert varieties, see \cite{MS}.) This allows us to talk about Weil Divisors on Schubert Varieties.

\begin{prop}
    The class group of a Schubert variety is generated by the codimension one Schubert cell closures.
\end{prop}

\begin{proof}
    
Because of the Bruhat decomposition, the open cell $C_w$ in a Schubert variety $X_w$ is an affine space of dimension $l(w)$. Therefore, the class group of a Schubert variety is generated by the codimension one Schubert subvarieties. It is in fact freely generated, as, if there is a rational function only having zeroes and poles of certain orders on the respective divisors, it is non-vanishing regular on the biggest affine open, which is not possible.
\end{proof}

One can ask the natural question about the behavior of the zero locus of various sections of these line bundles restricted to $X_w$.

\begin{prop}\label{Beigenvector}
Given a Schubert subvariety $X_w, w \in W$ of $G/B$, there exist an ample $G$-linearized line bundle on $X = G/B$  with section $\sigma$, such that restricting on the Schubert variety $X_w$, we get 
$$\text{div}(\sigma) = \sum_{{l(w^{\prime}) = l(w) -1}} c_{w^{\prime}}X_{w^{\prime}} $$ 
with all $c_{w^{\prime}}>0.$
    
\end{prop}

\begin{proof}
    
Let $L$ be an ample $G$-linearized line bundle on $G/B$ (which can be obtained by choosing an appropriate dominant weight). The restriction of $L$ to $X_w$ is an ample $B$-linearized line bundle, as  $X_w\subset G/B$ is $B$-stable. By ampleness, there exists a positive power $L^n$ and a global section in
$H^0(X_w,L^n)$ which vanishes along all the Schubert subvarieties $X_v\subset X_w$ of codimension $1$. More generally, given a subscheme $X \subseteq Y$, and an ample bundle $L$ on $Y$, for $n$ large enough, there are non-zero sections of $L^n$ which vanish on $X$: for an embedding to a projective space given by some $L^n$, $X$ is the vanishing locus of a collection of homogeneous functions. So, we can use the product to have a non zero section of a higher tensor power whose vanishing locus contains $X$.\\ Since $L^n$ is $B$-linearized, $H^0(X_w,L^n)$ is a $B$-module, and so is the subspace of sections vanishing identically along the $X_v$ as above. So this subspace contains a $B$-eigenvector $s$. Now, as $C_w$ is an open $B$-orbit in $X_w$, the section should be entirely non-vanishing on the open $C_w$, as it is non-vanishing at a point in it and is also a $B$-eigenvector. So finally, the divisor of $s$ is a Cartier divisor of the form $\sum_{{l(w^{\prime}) = l(w) -1}} c_{w^{\prime}}X_{w^{\prime}} $ with all $c_{w^{\prime}}>0$. 
\end{proof}
We have therefore accomplished the positivity of the coefficients above without deriving an explicit formula for a general group $G$. \begin{rmk}
    However, there is also Chevalley's formula that we can use (\cite{bla}):
\[ \text{div}(p_{w(\lambda)}|_{X_w}) = \sum \langle \lambda, \hat{\beta} \rangle X_{ws_{\beta}} \]

where the sum is over all $\beta \in \phi^{+}$ such that $X_{ws_{\beta}}$ is a divisor in $X_w$ ($p_{w(\lambda)}$ is a $T$–eigenvector of weight $-w(\lambda)$ in $H^0(G/B,L_{\lambda})$). The above is written for a dominant weight $\lambda$ and thus can also be used for getting positive coefficients for every Schubert divisor in $X_w$.
\end{rmk} 
For our purposes however, we need more general results about sections of line bundles and their vanishing when restricted to a union of Schubert Varieties, without precisely knowing the order of vanishing. Then we pull back those sections across $G \to G/B$ to get functions with analogous properties. This is where we follow the proof of \ref{Beigenvector} instead of using Chevalley's formula:
\begin{Cor}\label{functions}
Fix a $w\in W$ there exists a regular function $f_{w}$ on $\overline{BwB}$, such that it vanishes with non-zero multiplicity on each Bruhat cell $BvB$ that appears in $\overline{BwB} \backslash BwB = \bigcup_{v < w} BvB$ and is nowhere vanishing on the open cell $BwB$. \\
 More generally, we have the following: Consider Weyl group elements $w_i$ of the same length (so that $Bw_iB$ are of the same dimension) then there exists a regular function on $\bigcup_{i}\overline{Bw_iB}$ which vanishes on $\bigcup_{i}\overline{Bw_iB} \backslash \bigcup_{i}B{w_i}B$ and is nowhere vanishing on the open $\bigcup_{i}B{w_i}B$. 
\end{Cor}

\begin{proof}
Since the line bundle $L$ on $G/B$ as constructed above is $G$-linearized, we have that it pulls back to
the trivial bundle across the quotient map $G \to G/B$. So the pull back of the section $\sigma$ from \ref{Beigenvector} is a regular function on $G$.\\
Secondly, $\overline{BwB}$ is the inverse image of $X_w$ across this map, and the section $\sigma$ satisfies:
\[  \text{div}(\sigma) = {\sum}_{l(w^{\prime}) = l(w)-1} c_{w^{\prime}}X_{w^{\prime}}\]
with all $c_{w^{\prime}}>0$. Let’s denote by $f_w$ the regular function restricted to $\overline{BwB}$. By construction, it vanishes
on each Bruhat cell that appears in $\overline{BwB}\backslash BwB$ and is nowhere vanishing on $BwB$.

For the case of multiple Bruhat cells of the same dimension, we do the following:  

Assume that the set of Weyl group elements of a fixed length $k$ is indexed by the set $S: = \{1, \cdots m \}$. For $i \in S$, let $Y_i$ be the boundary of $X_i$ (the complement of the
open Bruhat cell $C_i$). Let $Z_i$ be the union of $Y_i$ and the $X_j$ with $j \neq i$ (the complement of $C_i$ in $\cup_{k \in S } X_{k}$).
Consider the homogeneous ideal of $Z_i$ in $X$ (in the graded ring of sections of a very ample
$G$-linearized line bundle). This is a graded $B$-stable ideal, and hence
contains some $B$-eigenvector $f_i$ which is homogeneous of some degree $n_i$.
The zero set of $f_i$ is exactly $Z_i$ (since $f_i$ vanishes on $Z_i$ and
vanishes nowhere on its complement, a unique $B$-orbit).
Replacing the $f_i$ with suitable powers, we may assume that they all
have the same degree $n$. As $f_i$'s are non-vanishing on $C_i$ and vanish everywhere else, we have that the sum $f$ of the $f_i$ has $\cup_{i \in S} Y_i$ as its zero set.\\
Finally, we pull-back the section across the map $G \to G/B$ to get a function like we wanted.
\end{proof}

\section {Construction of the anti-involution}\label{section antihom}
In this section we construct the anti-involution $\bm{\Psi}:G\to G$. Let us first recall the notion of non-degenerate local systems on $U$ and its relation to the notion of pinnings of reductive groups $G$.

Given any (possibly non-commutative) connected unipotent group $H$, we have a (possibly disconnected) commutative perfect unipotent group $H^{\ast}$, known as the Serre dual of $H$, which is the moduli space of multiplicative local systems on $H$ (see \cite{bo}, A.1, A.12 for details). For the natural map $H \to H^{\ab}$ from the $H$ to its abelianization, we have the induced map $(H^{\ab})^{\ast} \to H^{\ast}$ between the duals through pulling back multiplicative local systems across the abelianization map. We have that this map identifies $(H^{\ab})^{\ast}$ with $(H^{\ast})^0$, the neutral connected component of $H^{\ast}$.

Let us fix a non-trivial additive character ${\psi} : \mathbb{F}_p \to \overline{\mathbb{Q}_{\ell}}^{\times}$. This defines for us the Artin-Schreier local system $\mathcal{L}_{{\psi}}$ on $\Ga$ and gives us an identification of the Serre dual ${\Ga}^{\ast} \cong \Ga$ after passing to the perfectizations (see \cite{bdr} for details).  In particular, the moduli space of multiplicative local systems on $\Ga$ gets identified with the perfectization of $\Ga$ once we fix the non-trivial additive character $\psi$.

\begin{rmk}\label{rk:perfectizations}
Since the Serre dual of a connected unipotent group is only well defined as a perfect unipotent group scheme, we will often implicitly pass to perfectizations of the algebro-geometric objects. Note that passing to the perfectizations does not change the categories of $\Qlcl$-complexes.
\end{rmk}

In this article, we will be interested in non-degenerate multiplicative local systems $\mathcal{L}$ on the maximal unipotent subgroup $U$ of the given reductive group $G$ over $k$. They happen to be coming from pull-back from $U^{\text{ab}}$. (We first pull-pack the local system $\boxtimes_{\alpha \in \Delta} \mathcal{L}_{\alpha}$ via the map $U^{\text{ab}} \to \prod_{ \alpha \in \Delta} U_{\alpha}$; this map is an isomorphism if $p$ is a good prime for $G$.)
\begin{defn}\label{ndmls} (Non-degenerate local systems on the maximal unipotent subgroup of a reductive group.) Let $G$ be a reductive group $G$ over $k$ (algebraically closed of characteristic $p>0$), $U$ a maximal unipotent subgroup and $T$ a maximal torus normalizing $U$. A multiplicative local system $\mathcal{L}$ on $U$ is called non-degenerate if $\mathcal{L}$ is obtained as a  tensor product of pull-backs of non-trivial multiplicative local systems $\mathcal{L}_{\alpha}$ on $U_{\alpha}$, where $U_{\alpha}$ is the simple root subgroup corresponding to $\alpha \in \Delta$. More precisely, for the canonical morphism $U \xrightarrow {\pi} \prod_{\alpha \in \Delta} U_{\alpha }$, we should have: 
$\mathcal{L} \cong \pi^{\star} \left(\boxtimes_{\alpha\in \Delta} \mathcal{L}_{\alpha} \right)$, where each $\L_\alpha$ is non-trivial.
\end{defn}
Note that $T$ acts transitively on the set of all non-degenerate multiplicative local systems on $U$ and that the stabilizer in $T$ of any non-degenerate multiplicative local system on $U$ is the center of the group $Z\subset T$.

Given a reductive group $G$ and a maximal torus $T$, we use the notation 
$$R(G,T):= (X^{\ast}(T), \bm{\phi}(G,T), X_{\ast}(T), \bm{\phi}(G,T)^{\vee})$$ 
for the root datum. If we also fix a Borel subgroup $B$ containing $T$, then in the dual lattices $X^{\ast}(T)$ and $X_{\ast}(T)$ we have the finite subsets
$\Delta \subset \bm{\phi}(G,T) \subset X^{\ast}(T) $(the simple roots inside the roots) and analogously, we have $\Delta^{\vee} \subset \bm{\phi}(G,T)^{\vee}\subset X_{\ast}(T)$. The root datum along with the choice of simple roots and coroots, $(R(G,T), \Delta, \Delta^{\vee})$ is known as a based root datum.

\begin{defn} (Pinnings on Reductive Groups, Definition 1.5.4 in \cite{bc}) A pinning of $(G, T, B)$ is the specification of a $T$-equivariant isomorphism $p_{\alpha}: \Ga \isom U_{\alpha}$ for each $\alpha \in \Delta$, i.e.  $tp_{\alpha}(v)t^{-1} = p_{\alpha}(\alpha(t)v)$ for all $t\in T,v\in \Ga$. The data $(G, T, B, \{p_{\alpha}\})_{\alpha\in \Delta}$ is a pinned connected reductive group. We have the obvious notion of morphisms between pinned connected reductive groups.
\end{defn}     
    
\begin{prop} \label{pinningthm}(Proposition 1.5.5 in \cite{bc})
For pairs of pinned reductive groups, $(G,T,B,\{{p_{\alpha}}\}_{\alpha \in \Delta})$ and $(G^{\prime},T^{\prime},B^{\prime}, \{p_{{\alpha}^{\prime}}\}_{{\alpha}^{\prime}\in \Delta^{\prime}})$, the map
\[\xymatrix{
 \text{Isom}(G,T,B,\{{p_{\alpha}}\}_{\alpha \in \Delta}), (G^{\prime},T^{\prime},B^{\prime}, \{p_{{\alpha}^{\prime}}\}_{{\alpha}^{\prime}\in \Delta^{\prime}}))  \ar[d] \\
 \text{Isom}((R(G,T), \Delta, \Delta^{\vee}), (R(G^{\prime},T^{\prime}), \Delta^{\prime}, {\Delta^{\prime}}^{\vee}))
}\] 
is bijective. In particular, if $f$ is an automorphism of $(G, T, B)$ that is the
identity on $T$ and on the simple positive root groups then $f$ is the identity on
$G$, and a choice of pinning $\{{p_{\alpha}}\}_{\alpha \in \Delta}$ on $(G, T, B)$ defines a homomorphic section to the quotient map $Aut(G) \to  Out(G)$, where the latter group of outer automorphisms can be identified with the group  $\Aut(R(G,T), \Delta, \Delta^{\vee})$ of automorphisms of the based root datum. 
\end{prop}

We also state the following lemma which explains that for a connected reductive group $G$ over $k$ with a Borel pair $T\subset B$, {the choice of a non-degenerate local system on the maximal unipotent $U$ is equivalent to a choice of a pinning on the reductive group}. Fixing a non-trivial character $\psi : \mathbb{Z}/ p\mathbb{Z} \to \overline{\mathbb{Q}}_{\ell} ^{\times}$,  we get the Artin-Schrier local system $\mathcal{L}_{\psi}$ on $\Ga$.

\begin{lemma}\label{pinfromls}
Given a reductive group $G$ with a maximal unipotent $U$, a Borel $B$ and a maximal torus $T$ as in the introduction, the choice of a non-degenerate multiplicative local system $\mathcal{L}$ on $U$ is equivalent to a choice of pinning of $(G,T,B)$.     
\end{lemma}
\begin{proof}
By Definition \ref{ndmls}, we have an isomorphism $\mathcal{L} \cong \pi^{\star} \left(\boxtimes_{\alpha\in \Delta} \mathcal{L}_{\alpha} \right)$, where each $\L_\alpha$ is non-trivial on $U_{\alpha}$ which is the simple root subgroup corresponding to $\alpha \in \Delta$ and $\pi$ is the canonical morphism $U \xrightarrow {\pi} \prod_{\alpha \in \Delta} U_{\alpha }$. We have that the Serre dual of $\Ga$ is identified with $\Ga$ equivariant with the obvious action of $\Gm$. Therefore, as $\mathcal{L}_{\alpha}$ is non-trivial, we get the unique identification $p_{\alpha}: \Ga \to U_{\alpha}$ such that $p_{\alpha}^{\ast}\mathcal{L}_{\alpha}$ is isomorphic to $\mathcal{L}_{\psi}$ as multiplicative local systems on $\Ga$. For all $\alpha \in \Delta$, these identifications are clearly $T$-equivariant giving us the pinning  $(G, T, B, \{p_{\alpha}\})_{\alpha\in \Delta}$. It is now also clear how to manufacture the unique non-degenerate local system from the data of a pinning.   
\end{proof}

\begin{rmk}
The construction above relies on the fact that the local system $\mathcal{L}$ is non-degenerate, in the sense that $\mathcal{L}$ is obtained from individual non-trivial $\mathcal{L}_{\alpha}$ on the simple roots subgroups $U_\alpha$ as in the definition. The condition on a non-degenerate local system also enforces $\mathcal{L}$ to be lying in the connected component of the moduli space $U^*$ of multiplicative local systems on $U$ as mentioned in the introduction to this section. 
\end{rmk}

We now construct a canonical  involution $\tau: G \to G$ given a non-degenerate local system $\L$ on $U$, or equivalently a pinning of $(G,T,B)$:

\begin{lemma}\label{tauconstruction}
With $G$, $B=TU$ as in the lemma above, let $\L$ be a non-degenerate multiplicative local system on $U$. Then there is a unique automorphism $\bm\tau:G \to G$ such that
\begin{altenumerate}
    \item $\bm\tau(T)=T$ and $\bm\tau|_T= -\ad(w_0):T\to T$, where $w_0\in W$ is the longest element, i.e. for all $t\in T$ we have $\bm\tau(t)=\ad(w_0)(t^{-1})$.
    \item $\bm\tau(U)=U$, and hence $\bm\tau(U^{-})=\Uop$.
    \item $\bm\tau^*\L\cong \L^{-1}$, the non-degenerate multiplicative local system inverse, or dual to $\L$.
\end{altenumerate}
Moreover, the $\bm\tau$ as above is an involution, i.e. satisfies $\bm\tau^2=\id_G$.
\end{lemma}
\begin{proof}
For $w_0$, the longest element in the Weyl group, we consider the automorphism of root data induced by $-w_0:X^*(T)\to X^*(T)$, $\alpha \rightarrow -w_0(\alpha)$. This sends the positive roots to positive roots, the negative roots to the negative roots (as the longest element in the Weyl group flips them around) and hence also permutes the set of simple roots and is an automorphism of the based root datum. Note that the inverse multiplicative local system $\L^{-1}$ on $U$ is also non-degenerate and hence by Lemma \ref{pinfromls} determines another pinning on $(G,T,B)$. 
Hence by Proposition \ref{pinningthm} we get the unique $\bm\tau:(G,T,B)\to (G,T,B)$ as desired. It also clear by Proposition \ref{pinningthm} that $\bm\tau$ is an involution.
\end{proof}

\begin{prop}\label{psiprop}
    
Given a reductive group $G$, a fixed maximal unipotent subgroup $U$, maximal torus $T$, corresponding opposite unipotent $U^{-}$ and a non-degenerate multiplicative local system $\mathcal{L}$ on $U$, there is a unique anti-homomorphism $\bm{\Psi} : G \rightarrow G$ such that: 
\begin{altenumerate}
    \item $\bm{\Psi}(T) = T$, with $\bm{\Psi}$ acting on $T$ via conjugation by the longest element of the Weyl group, $w_0$.
    \item $\bm{\Psi} (U) = U$, and hence $\bm{\Psi} (U^{-}) = U^{-}$. 
    \item $\bm{\Psi} ^{\star} \mathcal{L} \isom \mathcal{L}$.  
\end{altenumerate}
Moreover, the $\bm\Psi$ as above is an anti-involution, i.e. $\bm{\Psi} ^2 = \text{id}$.
\end{prop}
\begin{proof}
Let $\bm{\Psi} := \bm{\tau} \circ \iota = \iota \circ \bm{\tau}$, where $\iota$ is the inverse map. We have $\bm{\Psi}^{\star} \mathcal{L} \cong  \mathcal{L}$ since $\iota^*(\L^{-1})\cong \L$.\\
The other properties of $\bm{\Psi}$ follow because they are true for $\tau$. By Lemma \ref{tauconstruction}, $\bm\Psi$ is the unique anti-automorphism as desired. It is also clear that it is an anti-involution.
\end{proof}

\begin{rmk}
    The involution $\alpha \to -w_0(\alpha)$ will be identity if $-\Id$ lies in the Weyl group. This is the case for groups of type $B_n, C_n, D_n$ for $n$ even, $G_2, F_4, E_7, E_8$. It is not the identity in the case of $GL_n$.
\end{rmk}

\begin{example}
Consider $G=GL_n$, with $T\subset B=TU$ being the standard maximal torus and Borel subgroup (of upper triangular matrices). In this case we have the projection $U\to U^{\ab}\cong \Ga^{n-1}$ which sends an upper triangular unipotent matrix to its entries above the diagonal. Consider the non-degenerate local system $\mathcal{L}$ on $U$ defined as the pull back of the Artin-Schreier local system $\mathcal{L}_{\psi}$ on $\Ga$ along the composition: $U \to U^{\text{ab}}\isom \Ga^{n-1}\xrightarrow{\Sigma} \Ga$, where $\Sigma$ is the summation map. In this case, the resulting anti-involution $\bm{\Psi}$ on $GL_n$ is given by the operation of transposition along the anti-diagonal. Note that the same anti-involution would work as long as we choose any non-degenerate local system $\L$ on $U$ which is fixed by this anti-involution.
\end{example}

\begin{defn}\label{nwldef}
For $w\in W,$ we have the coset $Tw \subseteq N$ of the normalizer $N$ of the torus $T$. Let $N_{w,\mathcal{L}} \subseteq Tw$ be the (possibly disconnected) reduced closed subscheme (see Remark \ref{rk:closedsubscheme} below) consisting of points $n \in Tw$ such that $\mathcal{L}_{U\cap {^w U}}\cong {^n\mathcal{L}}_{U\cap {^w U}}$, where $^n\mathcal{L}$ is the multiplicative local system on ${}^nU= { }^wU$ obtained by conjugating $\L$ by $n$. Define $N_\L\subset N$ to be the union $\bigcup\limits_{w\in W} N_{w,\L}$.
\end{defn}
\begin{rmk}\label{rk:closedsubscheme}
We have that $N_{w,\mathcal{L}} \subseteq Tw$ is in fact a closed subscheme of $Tw$: Consider the morphism (of perfect schemes after passing to the perfectizations) $m_{w,\mathcal{L}}:Tw \to (U\cap {^w U})^{\ast}$ (the Serre Dual of the unipotent group $U\cap {{^w}U}$) defined as $tw \mapsto {\mathcal{L} \otimes (^{tw}\mathcal{L})^{-1}}_{|U\cup {^w}U}$. Hence we see that $N_{w,\mathcal{L}} \subset Tw$ is the fiber over the identity $1 \hookrightarrow (U\cap {^w}U)^{\ast}$ and therefore, it is a closed perfect subscheme.    
\end{rmk}
\begin{rmk}
We see that the condition for an $n \in Tw$ to lie in $N_{w,\mathcal{L}}$ amounts to getting isomorphic local systems upon restrictions to every $U_{\alpha}$ lying in the intersection $U\cap {^w}U$. For example using this, one sees that: for $w=w_0$, we get $N_{w_0,\mathcal{L}} = Tw_0$ and for $w=e$ (the identity element in $W$), we get $N_{e,\mathcal{L}} = Z$, as $Z = \bigcap\text{ker } \alpha$.      
\end{rmk}

\begin{prop}\label{support}
The anti-involution $\bm{\Psi}$ fixes the subvariety $N_\L\subset G$ pointwise, i.e. $\bm{\Psi} (n) = n, \ \forall n \in N_\L$.
\end{prop}
\begin{proof}
  
After setting up all the results from the previous proposition, the above follows from rewriting the proof of theorem 8.1.3 in \cite{car} in our geometric setting. We include the proof here for the benefit of the reader. The proof is  simpler in our geometric setting which is a feature of working with algebraic groups over an algebraically closed field.

The proof is in three main steps: Consider the quotient map $\pi: N \to W$. We first characterize the image $\pi(N_{\mathcal{L}})\subset W$. Secondly, we show that there is a system of coset representatives $n_w \in Tw \subset N$ for each $w \in W$, where $n_w = n_{w^{\ast}}$ (${w^{\ast}}$ being defined as the image of $w$ under $\bm{\Psi}$). Finally, using the first characterisation we shall have that if $n_w \in N_{\mathcal{L}}$, then $w = w^{\ast}$ and hence $\bm{\Psi}(n_w) = n_w $. Using these coset representatives, we shall more generally be able to say that $\bm{\Psi}(n) = n$ for all $n \in N_{\mathcal{L}}$.

Let's first fix an $n \in N_{\mathcal{L}}$ and let $w \in W$ be its image under $\pi$. From the above remarks, we have $e_{\mathcal{L}} \ast \delta_n \ast e_{\mathcal{L}} \neq 0$.
However, $e_{\mathcal{L}} \ast \delta_{n} \ast e_{\mathcal{L}} = \prod_{{\alpha} >0} e_{\mathcal{L}_{\alpha}} \ast \delta_n \ast \prod_{\alpha >0}e_{\mathcal{L}_\alpha} =  \prod_{{\alpha} \neq \beta >0} e_{\mathcal{L}_{\alpha}} \ast e_{\mathcal{L}_{\beta}} \ast \delta_n \ast e_{\mathcal{L}_{\beta}} \ast \prod_{\alpha \neq \beta >0}e_{\mathcal{L}_\alpha} $. So, for $\beta$ a simple root as we have $e_{\mathcal{L}_{\beta}} \ast \delta_n \ast e_{\mathcal{L}_{\beta}} \neq 0  \iff e_{\mathcal{L}_{\beta}} \ast e_{^n\mathcal{L}_{\beta}} \neq 0 $ so if $w(\beta)>0$ then it is also simple ($\mathcal{L}_{\gamma}$ is non-trivial only when $\gamma$ is a simple root).

Therefore, $w$ has the property that if for a positive simple root $\alpha$, $w(\alpha)>0$, then $w(\alpha)$ is also simple. So let $K$ be the set of simple roots of $G$ such that 
$w(\alpha)>0$.     
Consider the element $w(w_0)_{K}$ (with $(w_0)_{K}$ being the longest element in the Weyl subgroup corresponding to the set of simple roots $K$).

Say $\alpha$ is simple and $\alpha \in K$, i.e. $w(\alpha)>0$. Now, $w(w_0)_K (\alpha) = -w(-(w_0)_{K}(\alpha)).$ However, $-(w_0)_{K}(\alpha)$ is an element in $K$, and so $-w(-(w_0)_{K}(\alpha))$ is the negative of an element in $K$ and hence $<0$.\\
If $\alpha$ is simple but $\alpha \notin K$, then $w(w_0)_{K}(\alpha) = w(\alpha + \beta) = w(\alpha) + w(\beta)$ (here $\beta$ is a linear combination of $K$-roots and hence $w(\beta)$ is a linear combination of $w(K)$-roots). However $w(\alpha)$ is negative (by assumption) and not in $w(K)$. Therefore, $w(w_0)_{K}(\alpha)<0$. We have that $w(w_0)_{K}$ turns all positive simple roots to negative roots and so, $w(w_0)_{K} = w_0$.

Moving on to the second step, for any $w' \in W$, let $w'^{\ast}$ denote the image of $w'$ under $\bm{\Psi}$. Our goal now is to find $n_{w'}\in Tw' \subset N$ for each $w'\in W$ such that $\bm{\Psi}(n_{w'}) = n_{{w'}^{\ast}} \forall w' \in W$ satisfying some additional properties as will be stated.\\
We start with simple roots and consider the equivalence relation on simple roots generated by:
\begin{enumerate}[label=(\alph*)]
    \item $\alpha_1$ is equivalent to $\alpha_2$ if there exists $n \in N$ with $nU_{\alpha_1}n^{-1} = U_{\alpha_2}$ (and in this case we can choose $n$ such that $^n\mathcal{L}_{U_{\alpha_1}} = \mathcal{L}_{U_{\alpha_2}}$).
    \item $\alpha_1$ is equivalent to $\alpha_2$ if $\bm{\Psi}(U_{\alpha_1}) = U_{\alpha_2}$
\end{enumerate}
We now choose one root in each equivalence class. For such a simple root $\alpha$, let $G_{\alpha} \subset G$ be the associated reductive subgroup of semisimple rank $1$. Suppose first that $\alpha = \alpha^{\ast}$.  We have $\langle U_{\alpha}, U_{-\alpha} \rangle \subseteq U_{-\alpha}S \cup U_{-\alpha}Sn_{s_{\alpha}}U_{-\alpha} = G_{\alpha}$, where $S = T \cap \langle U_{\alpha}, U_{-\alpha} \rangle$ ($n_{s_{\alpha}}$ a coset representative of the reflection $s_{\alpha}$ corresponding to $\alpha$). Now, as $\bm{\Psi}$ is identity on $U_{\alpha}$ and $U_{-\alpha}$, along with the fact that $U_{\alpha} \subseteq U_{-\alpha}Sn_{s_{\alpha}}U_{-\alpha}$ we get an $x \neq 1 \in U_{-\alpha}Sn_{s_{\alpha}}U_{-\alpha}$ which is fixed by $\bm{\Psi}$. Let $x = x_1 n_{s_{\alpha}} x_2$ ($n_{s_{\alpha}}$ is uniquely determined by $x$ in this way), now $\bm{\Psi}(x) =x \implies \bm{\Psi}(n_{s_{\alpha}}) = n_{s_{\alpha}}$.

Now suppose $\alpha \neq \alpha^{\ast}$.
We proceed similarly as in the $\alpha = \alpha^{\ast}$ case. Consider an element $x \neq 1 \in U_{\alpha}$ and as $\bm{\Psi}(U_{\alpha}) = U_{\alpha^{\ast}}$, let $x^{\ast} = \bm{\Psi}(x) \in U_{{\alpha}^{\ast}}.$ We have the inclusion $U_{\alpha} \subseteq U_{-\alpha}Sn_{s_{\alpha}}U_{-\alpha}$ and its image inclusion $U_{{\alpha}^{\ast}} \subseteq U_{-\alpha^{\ast}}\bm{\Psi}(S)n_{s_{\alpha}^{\ast}}U_{-\alpha^{\ast}}$, where we also have $\bm{\Psi}(S) = T \cap \langle U_{\alpha^{\ast}}, U_{-\alpha^{\ast}} \rangle$.

So following the previous case, $x = x_1n_{s_\alpha}x_2$, we get a corresponding $\bm{\Psi}(x) = \bm{\Psi}(x_2)\bm{\Psi}(n_{s_{\alpha}})\bm{\Psi}(x_1)$ and because of the uniqueness of the decomposition, we get a uniquely determined $n_{s_{\alpha}^{\ast}} = \bm{\Psi}(n_{s_{\alpha}}).$

Now that we have determined $n_{s_{\alpha}}$ for one element in every equivalence class, we can go ahead with defining $n_{\alpha}$ for all $\alpha$: by restricting to a particular equivalence class we can use the relations $(a)$ and $(b)$ to define all the other $n_{s_{\alpha}}'s$ in the natural way. What is crucial here is that the result is independent of the composition maps that we get out of $(a)$ and $(b)$. This is reliant on the fact that $\bm{\Psi}$ and conjugation by $n$ both preserve $\mathcal{L}$ which in turn implies that two different ways of going from $U_{\alpha}$ to $U_{\beta}$ (for two roots $\alpha$ and $\beta$ in an equivalence class) will be the same. \\
We also take note that by virtue of this construction, $nn_{\alpha}n^{-1} = n_{w(\alpha)}$.

Now for $w = s_{\alpha_1}\cdot s_{\alpha_2} \cdots s_{\alpha_k}$ (a minimal expression in terms of simple reflections), define $n_w = n_{s_{{\alpha}_1}}\cdot n_{s_{{\alpha}_2}}\cdots n_{s_{{\alpha}_k}}$. We also define $w^{\ast}$ as $s_{\alpha_{k}}^{\ast} \cdots s_{{\alpha_2}}^{\ast} \cdot s_{{\alpha_1}}^{\ast}.$\\
We have $\bm{\Psi} (n_w) = \bm{\Psi}(n_{s_{{\alpha}_1}}\cdot n_{s_{{\alpha}_2}}\cdots n_{s_{{\alpha}_k}}) = n_{s_{{\alpha}_k}^{\ast}}\cdot n_{s_{{\alpha}_2}^{\ast}}\cdots n_{s_{{\alpha}_1}^{\ast}} = n_{w^{\ast}}$ (This will be independent of the choice of a minimal representation for $w$ as mentioned in \cite{car}).\\
At the same time, $w^{\ast}= {s_{{\alpha_{k}}}^{\ast} \cdots s_{{\alpha_2}}^{\ast} \cdot s_{{\alpha_1}}^{\ast}} = w_{0}s_{{\alpha_{k}}}w_{0}^{-1} \cdots w_{0}s_{\alpha_1}w_{0}^{-1} = w_{0}w^{-1}w_{0}$.\\
We have proved that for $n$ such that $e_{\mathcal{L}} \ast \delta_{n} \ast e_{\mathcal{L}} \neq 0$, with $w = \pi(n)$, we have $w = w_0(w_0)_{K}$ ($K$ being the set of simple roots that stay positive and in fact simple under $w$).\\
So for such a $w$, $w^{\ast} = w_{0}w^{-1}w_{0}^{-1} = w_0 (w_0)_{K} = w $.
Hence $\bm{\Psi}(n_w) = n_{w}$ for such $w's$.
We also have $w_0(w_0)_{K}w_{0} = (w_0)_{L}$ for some set of simple roots $L$.\\
Define $(n_0)_{K}, (n_0)_{L}$ by $(n_0)_{K} = n_{^{w_0}K}, (n_0)_{L} = n_{^{w_0}L}, n_{0} = n_{w_0}$.\\
As $w_0 = w(w_0)_{K} = (w_0)_{L}w$,
we have $n_0 = n_w(n_0)_{K} = (n_0)_{L}n_w.$ 
Thus, $n_w(n_0)_{K}n_w^{-1} = (n_0)_{L}$.\\
The idea now is to compare with $n(n_0)_Kn^{-1}$; let $(w_0)_{K} = s_{\beta_1}\cdots s_{\beta_k} \implies (n_0)_{K} = n_{\beta_1}\cdots n_{\beta_k}.$\\
So we have $n(n_0)_Kn^{-1} = n_{w(\beta_1)}\cdots n_{w(\beta_k)}$. However, $(w_0)_{L} = s_{w(\beta_1)} \cdots s_{w(\beta_k)}$ is a reduced expression for $(w_0)_L$ and so $(n_0)_{L} = n_{w(\beta_1)} \cdots n_{w(\beta_k)}$ (as $nn_{\alpha}n^{-1} = n_{w(\alpha)}$ for a simple root $\alpha$) .\\
So finally, we have that $n(n_0)_{K}n^{-1} = n_{w}(n_0)_{K}n_{w}^{-1}$, and so if $n = n_wt$, for some $t\in T$ ($n$ and $n_w$ are in the same coset), then we have $t(n_0)_{K}t^{-1} = (n_0)_K$.
Hence it follows that $\bm{\Psi}(n) = \bm{\Psi}(n_wt) = \bm{\Psi}(t)\bm{\Psi}(n_w) = n_0tn_0^{-1}n_w = n_w(n_0)_{K}t(n_0)_{K}^{-1} = n_wt = n$.

\end{proof}

\section{Construction of the Symmetric Monoidal Structure using Beilinson's gluing of perverse sheaves}\label{gluingconstr}

Let us first introduce the category of bi-Whittaker sheaves on the group $G$. It will be
convenient to recall some facts and notation from the theory of character sheaves on unipotent groups
developed by Boyarchenko-Drinfeld in \cite{bdr}. Recall that we have a monoidal structure on the category $\mathscrsfs{D}(G)$ which comes from convolution with compact support. For a multiplicative local system $\mathcal{L}$ on any unipotent algebraic group $H$, let $e_\mathcal{L} := \mathcal{L} \otimes \omega_H = \mathcal{L}[2 \text{dim} H](\text{dim} H)$ be the corresponding closed idempotent (see \cite{bdr} for
details) in $\mathscrsfs{D}_H(H)$. If $\mathcal{L}$  is the trivial local system on $H$, we will often denote the corresponding idempotent simply by $e_H \in \mathscrsfs{D}_H(H)$.

\begin{defn}\label{heckedefn}
The bi-Whittaker category is defined to be the triangulated monoidal category $e_{\mathcal{L}}\mathscrsfs{D}(G)e_{\mathcal{L}} \subseteq \mathscrsfs{D}_{U}(G) \subseteq \mathscrsfs{D}(G)$,
i.e. it is the Hecke subcategory corresponding to the closed idempotent $e_{\mathcal{L}}$ in the terminology of \cite{bd}.

In particular, the objects are the objects in $\mathscrsfs{D}(G)$ which are of the form $e_{\mathcal{L}} \ast \mathcal{F} \ast e_{\mathcal{L}}$ and $e_{\mathcal{L}}\mathscrsfs{D}(G)e_{\mathcal{L}}$ is defined to be the full subcategory formed by such objects. The object $e_{\mathcal{L}} \in e_{\mathcal{L}}\mathscrsfs{D}(G)e_{\mathcal{L}}$ is the unit object.
   
\end{defn}

Our main goal in this paper is the construction of a structure of a symmetric monoidal category for $e_{\mathcal{L}}\mathscrsfs{D}(G)e_{\mathcal{L}}$ (following which, we show its canonicity). Let's quickly outline the main ideas before going into the details:

We observe that using the antihomomorphism $\bm{\Psi}$ constructed in the previous section, we have an endofunctor $\bm{\Psi}^{\ast}$ for the category $e_{\mathcal{L}}\mathscrsfs{D}(G)e_{\mathcal{L}}$ (this follows from the properties established in \ref{psiprop}). 

Our first goal will be to construct a canonical natural isomorphism to the identity functor. So, in particular we want $\bm{\Psi}^{\ast} \mathcal{F} \cong \mathcal{F}$ \textbf{functorially}. 

\begin{rmk}\label{flip}
Given the above construction, we have the following natural transformations: $\mathcal{F} \ast \mathcal{G}\cong\bm{\Psi}^{\star}(\mathcal{F} \ast \mathcal{G}) \cong \bm{\Psi}^{\star}({\mathcal{G}}) \ast \bm{\Psi}^{\star}(\mathcal{F})\cong \mathcal{G} \ast \mathcal{F}$ (The second natural isomorphism is due to the fact that $\bm{\Psi}$ is an antihomomorphism, the others will be induced through the natural transformation that we will construct). This composition will induce the symmetric monoidal structure that we want. (The remaining conditions for this will be verified in Section \ref{symmeq}).
\end{rmk}
\begin{lemma}\label{heckesupp}
The support for sheaves in the bi-Whittaker Category $e_{\mathcal{L}}\mathscrsfs{D}(G)e_{\mathcal{L}}$ lies in the subscheme $UN_{\mathcal{L}}U \subset G.$   
\end{lemma}
\begin{proof}
 Note that from the Bruhat decomposition we obtain $G=UNU = \bigcup_{w\in W} U\cdot Tw \cdot U$, where $Tw \subset N$ are the various $T$-cosets in the normalizer $N$ of the maximal torus $T$.\\ 
We therefore have to show that the double cosets in $U\backslash G/U$ which can lie in the support of objects of $e_{\mathcal{L}} \mathscrsfs{D}(G)e_{\mathcal{L}}$ are given by $N_{\mathcal{L}} \subset N$. By a standard argument from Mackey theory, objects of $e_{\mathcal{L}} \mathscrsfs{D}(G)e_{\mathcal{L}}$ can only be supported on those $g \in G$ such that $e_{\mathcal{L}} \ast \delta_g \ast e_{\mathcal{L}} \neq 0$ (refer \cite{bd}, \cite{td}) i.e. those $g$ such that $e_{\mathcal{L}} \ast {^g\mathcal{L}} \neq 0$.

Let us now suppose that $g \in N$ since $N$ represents all the $U - U$-double cosets.

We have that the reduced closed (possibly disconnected) subscheme of $Tw$ formed by $n \in Tw$ such that $e_{\mathcal{L}} \ast \delta_n \ast e_{\mathcal{L}} \cong e_{\mathcal{L}} \ast e_{^n\mathcal{L}}\ast \delta_n \neq 0$ is exactly $N_{w,\mathcal{L}}$ as defined in \ref{nwldef}. This follows from \cite{td}, Lemma 2.14: $e_{\mathcal{L}} \ast e_{^n\mathcal{L}} \neq 0$ is equivalent to $\mathcal{L}_{|U\cap {{^n}U}} \cong {^n}\mathcal{L}_{|U\cap {{^n}U}}$.

Therefore we have that the support of an object in $e_{\mathcal{L}}\mathscrsfs{D}(G)e_{\mathcal{L}}$ must lie in $\bigcup_{w\in W} N_{w,\mathcal{L}} = N_{\mathcal{L}}.$

\end{proof}
\begin{lemma}\label{dictionary}
Let $e_{\mathcal{L}}\mathscrsfs{D}(BwB)e_{\mathcal{L}}$ be the full subcategory formed by objects in $e_{\mathcal{L}}\mathscrsfs{D}(G)e_{\mathcal{L}}$ which have support lying in the Bruhat cell $BwB \subset G$. There is an equivalence of categories: $e_{\mathcal{L}}\mathscrsfs{D}(BwB)e_{\mathcal{L}} \cong \mathscrsfs{D}(N_{w,\mathcal{L}})$.
\end{lemma}
\begin{proof}
We have the equality for the Bruhat cell $BwB = U \times Tw \times U_w$, where $ U_w := U/ {^{w}U} \cap U$ 
(the identification is using the multiplication operation of the group $G$). In light of the previous Lemma, we have the following identifications:\\ 
$\mathscrsfs{D}(N_{w,\mathcal{L}}) \xrightarrow[]{\cong} e_{\mathcal{L}}\mathscrsfs{D}(U\cdot Tw \cdot U)e_{\mathcal{L}} = e_{\mathcal{L}}\mathscrsfs{D}(BwB)e_{\mathcal{L}}$ given for $C \in \mathscrsfs{D}(N_{w,\mathcal{L}})$ as (denoting by $m$ and $n$ the two equivalences)

\[ C \xmapsto{m} e_{\mathcal{L}} \ast C \ast e_{\mathcal{L}}.\]

The inverse functor is given by 
\[ \mathcal{F} \in e_{\mathcal{L}} \mathscrsfs{D}(G)e_{\mathcal{L}} \xmapsto{n} \mathcal{F}_{|N_{w,\mathcal{L}}}[-2\text{dim} U - 2\text{dim} U_w](-\text{dim} U - \text{dim}U_w)\]
\end{proof}
We take the first step towards the construction of a natural transformation as mentioned earlier:
\begin{lemma}\label{bruhatnt}
Fix a $w \in W$. There is a canonical natural transformation $\theta: \bm{\Psi}^{\ast} \to \text{Id}$ on  $e_{\mathcal{L}}\mathscrsfs{D}(BwB)e_{\mathcal{L}}$, the full subcategory of sheaves in $e_{\mathcal{L}}\mathscrsfs{D}(G)e_{\mathcal{L}}$ supported on the Bruhat cell $BwB$.   
\end{lemma}
\begin{proof}
From Proposition \ref{support}, we have that $\bm{\Psi}$ acts as identity on $N_{w,\mathcal{L}}$. In addition, $\bm{\Psi}^{\ast}$ is an anti-monoidal functor satisfying $\bm{\Psi}^{\ast} e_{\mathcal{L}} = e_{\mathcal{L}}$. In view of the equivalence in the Lemma above, it is enough to describe a natural transformation between the induced functors on $\mathscrsfs{D}(N_{w,\mathcal{L}})$. As $\bm{\Psi}^{\ast}$ is the identity functor there, we use the canonical identity map between the identity functors.    
\end{proof}
What remains now is to \textbf{functorially} upgrade this natural isomorphism for general objects in $e_{\mathcal{L}}\mathscrsfs{D}(G)e_{\mathcal{L}}$.
We want to use what we already have for sheaves restricted to a single Bruhat cell along with the Bruhat decomposition: noticing that Bruhat cells are themselves strata in this decomposition.\\
Let us first discuss the obvious approach towards solving this problem and see why this fails:
We can, for example consider a sheaf $\mathcal{F}$  in our bi-equivariant category with support intersecting the biggest Bruhat cell. We have the open-closed triangle (for $V = Bw_0B$, $Z = G\backslash Bw_0B = \cup_{w \neq w_0} BwB$ and $j$ and $i$ denoting the open and closed inclusions respectively):

\[\xymatrix{
 j_{!}j^{!}\mathcal{F} \ar[r] \ar[d]^{\cong} & \mathcal{F} \ar[r]\ar[d]^{\exists !?} & i_{\ast}i^{\ast}\mathcal{F} \ar[r]^{+} \ar[d]^{\cong} & j_{!}j^{!}\mathcal{F}[1]\ar[d]^{\cong} \\
 j_{!}j^{!}(\bm{\Psi}^{\ast}\mathcal{F}) \ar[r] & \bm{\Psi^{\ast}}\mathcal{F}\ar[r] & i_{\ast}i^{\ast}(\bm{\Psi}^{\ast}\mathcal{F}) \ar[r]^{+} & j_{!}j^{!}(\bm{\Psi}^{\ast}\mathcal{F})[1]  .
}\]

The difficulty here is that, even if via some induction argument we get $\theta$ for the closed complement with the square on the right commuting, giving us an isomorphism between $\mathcal{F}$ and $\bm{\Psi}^{\ast}\mathcal{F}$: this isomorphism is \textbf{not uniquely determined}.

Instead, there is a workaround using perverse sheaves and the perverse t-structure which is a Verdier self-dual t-structure on the derived category $\D(G)$. We first observe that the natural isomorphisms already set up for objects supported on a single Bruhat cell restrict to isomorphisms for the perverse sheaves in the category. The key point will be Beilinson's gluing of perverse sheaves across the Bruhat strata, which will functorially glue these isomorphisms as well.

The induction will therefore proceed on the dimension of the Bruhat cells that intersect the support of $\mathcal{F}$: Take $U = BwB, X = \overline{BwB}, Z = \overline{BwB} \backslash BwB$ (more precisely, we will consider a union of Bruhat closures of same dimension), we have from \ref{functions} that this can be put in Beilinson's Gluing setup. Notice that in this gluing procedure, although we are making a choice of a function $f_w$ on $\overline{BwB}$ for every $w \in W$, it will be clear that one still ends up with the same natural isomorphisms. This process of inductive gluing uniquely upgrades the natural isomorphisms already described for sheaves supported on a single cell. 

We shall follow the expositions in \cite{rr} and \cite{sm}. \cite{rr} deals with the setup over the complex numbers. Morel's exposition in \cite{sm} deals with the case over a positive characteristic base. The theorems as we shall recall them will be mainly from Reich's notes but with the added Tate twists that appear in the $\ell$-adic setup. We start with first briefly recalling the definition of the (unipotent) nearby cycles and the vanishing cycles functor. Following which, we mention the equivalence of categories that showcases gluing.
\subsection{Beilinson's gluing of Perverse Sheaves}

Consider the following setup where we have an open-closed decomposition $U \hookrightarrow X \hookleftarrow Z$ of a variety $X$ with respect to a regular function $f$ (realised as the non-vanishing and vanishing locus of $f$).
\[\xymatrix{
 V \ar[r] \ar[d] & X \ar[d]^{f} & Z \ar[l] \ar[d] \\
 \Gm \ar[r] & \mathbb{A}^{1} & 0 \ar[l] .
}\] 

In this section, we won't go over all the details of the functors related to Beilinson's gluing construction but only recall some of their important properties. First, fix a topological generator $T$ of the prime-to-$p$ quotient of $\pi_1^{\text{geom}} ({\Gm}_k , 1)$ (where $p = char(k)$). And, let $t: \pi_1^{\text{geom}} ({\Gm}_k , 1) \to \Z_\ell(1)$ be the usual surjective map in the $\ell$-adic case.
The unipotent nearby cycles functor $\Psi_f^{un}: Perv(V) \to Perv(Z)$  (to be defined in \ref{nearbydef}) is a functor from the category of perverse sheaves on the open $V$ to the category of perverse sheaves on $Z$.
\begin{rmk}\label{N}
The nearby cycles functor $\Psi_f$ admits an action of $\pi_1^{\text{geom}} ({\Gm}_k , 1)$ and we obtain the functor $\Psi^{un}_f$ where the action is unipotent. There is a unique nilpotent $N: \Psi^{un}_f \to \Psi^{un}_f (-1)$ such that $T = \text{exp}(t(T)N)$ on $\Psi^{un}_f K$, for $K$ a perverse sheaf on $Z$. The operator $N$ is usually called the "logarithm of the unipotent part of the monodromy".
\end{rmk}
Let $K_a$ be the vector space of dimension $a \ge 0$, with the action of the unipotent matrix $J^a = [\delta_{ij} - \delta_{i,j-1}]$, variant of a Jordan block of dimension $a$. Let $\mathcal{K}_a$ be the local system on $\Gm$ whose underlying space is $K^a$ and in whose monodromy action the $t$ chosen above acts by $J^a$. We will often abuse notation and write $f^{\ast}\mathcal{K}^a$ instead of $f_{|\Gm}^{\ast}\mathcal{K}^a$ to denote the pullback of the local system on $\Gm$.

For any $X$ an algebraic variety over $k$, 
let $\text{Perv}(X) \subseteq \mathscrsfs{D}(X)$ be the abelian category $\text{Perv}(X)$ of perverse sheaves.\\
In the rest of this section, $\mathcal{M}$ is any object of $\text{Perv}(V)$.
In \cite{rr} (where the author works over the complex numbers), the definition of the (unipotent) nearby cycles and vanishing cycles is first given through pullbacks and pushforwards with respect to the universal cover of $\Gm.$ Such a definition is of course not algebraic in nature, but it turns out to be giving the same sheaves as in Beilinson's construction, which we take as our definition in our $l$-adic case:

\begin{defn}\label{nearbydef}(Proposition 2.2 in \cite{rr}) Let $\alpha_a : j_{!}(\mathcal{M} \otimes f^{\ast}\mathcal{K}^a) \to j_{\ast}(\mathcal{M} \otimes f^{\ast}\mathcal{K}^a)$ be the natural map. Then there is an inclusion $\ker(\alpha^a) \hookrightarrow {\Psi}^{un}_f(\mathcal{M}),$ identifying the actions of $\pi_1^{\text{geom}} ({\Gm}_k , 1)$, which is an isomorphism for sufficiently large $a$.    
\end{defn}

We also have a functorial action of $\pi_1^{\text{geom}} ({\Gm}_k , 1)$ on ${\Psi}_f^{un}$ and we get a functorial exact triangle
\[{\Psi}_f^{un} \xrightarrow{N} {\Psi}_f^{un}(-1) \to i^{\ast}j_{\ast} \xrightarrow{+1}. \]

We again refer to \cite{rr} for the construction of the functor $\Xi_f^{un}: \text{Perv}(V) \to \text{Perv}(X)$, which Beilinson calls the "Maximal Extension Functor".
\begin{prop}(Proposition 3.1 in \cite{rr})
There are two natural exact sequences exchanged by duality: $\mathcal{M} \leftrightarrow D\mathcal{M}$:

\[ 0 \to j_{!}(M) \xrightarrow{\alpha_{-}} \Xi_f^{un}(M) \xrightarrow{\beta_{-}} {\Psi}_{f}^{un}(M) (-1) \to 0\]
\[ 0 \to {\Psi}_{f}^{un}(M)  \xrightarrow{\beta_{+}} \Xi_f^{un}(M) \xrightarrow{\alpha_{+}} j_{*}(M) \to 0\]
where $\alpha_{+} \circ \alpha_{-} = \alpha $ and $\beta_{-} \circ \beta_{+}= N$, where $\alpha$ is the canonical map $j_{!}(M) \to j_{\ast}(M).$ Note that we have $\mathbb{D}({\Psi}^{un}_f K)$ is canonically isomorphic to ${\Psi}^{un}_f (\mathbb{D}K)(-1)$ and $\mathbb{D}(\Xi^{un}_f(M))$ is canonically isomorphic to $\Xi^{un}_f(\mathbb{D}M)$.   
\end{prop}
Take $\mathcal{M} = j^{\ast}\mathcal{F}$ for a perverse sheaf $\mathcal{F} \in \text{Perv}(X)$ in the above exact sequences. From the
maps in these two sequences we can form a complex:
\[ j_{!}j^{\ast} \mathcal{F} \xrightarrow{(\alpha_{-}, \gamma_{-})} \Xi_f^{un}(j^{\ast}\mathcal{F}) \oplus \mathcal{F} \xrightarrow{(\alpha_{+}, -\gamma_{+})} j_{\ast}j^{\ast}\mathcal{F} \hspace{30pt} (\ast)\]
Where $\gamma_{-}:j_{!}j^{\ast} \mathcal{F} \to \mathcal{F}$ and $\gamma_{+} \mathcal{F} \to j_{\ast}j^{\ast} \mathcal{F}$ are defined by left and right adjunctions $(j_{!},j^{\ast})$ and $(j^{\ast}, j_{\ast})$, with the property that $j^{\ast} (\gamma_{-})$ and $j^{\ast}(\gamma^{+}) = \text{Id}$.
We now recall some important consequences of defining the maximal extension functor, most notably, the construction of the vanishing cycles functor.
\begin{prop} (Proposition 3.2 in \cite{rr})
We have that $(\ast)$ above is in fact a complex: let $\bm{\tau}_f^{un}$ be its cohomology sheaf. Then $\bm{\tau}_f^{un}$ is an exact functor $\text{Perv}(X) \to \text{Perv}(Z)$ (called the vanishing cycles functor), and there are maps $u,v$ such that $v \circ u = N  $ as in the following diagram:
\[ {\Psi}_f^{un}(j^{\ast}\mathcal{F}) \xrightarrow{u} \bm{\tau}_f^{un}(\mathcal{F}) \xrightarrow{v} {\Psi}_f^{un} (j^{\ast} \mathcal{F})(-1)\]
where $N$ "is a logarithm for T" as explained in \ref{N}.
\end{prop}
After applying $j^{\ast}$ to $(\ast)$, we simply have: 
\[ \mathcal{M} \xrightarrow{(\text{id},\text{id})} \mathcal{M} \oplus \mathcal{M} \xrightarrow{(\text{id}, -\text{id})} \mathcal{M}\]
which is actually exact, so $j^{\ast} \bm{\tau}_f^{un}(\mathcal{F}) = 0,$  i.e., $ \bm{\tau}_f^{un}$ is only supported on $\mathcal{Z}$. To define $u$ and $v$: let $pr: \Xi_f^{un}(j^{\ast}\mathcal{F}) \oplus \mathcal{F} \to \Xi_f^{un}(j^{\ast}\mathcal{F})$  and set $u = (\beta_{+}, 0)$  in coordinates, and $v = \beta_{-} \circ pr$. Since $\beta_{-} \circ \alpha_{-}= 0$, $v$ factors through $\bm{\tau}_f^{un}(\mathcal{F})$, and we have $v \circ u = \beta_{-} \circ \beta _{+} = N$.

\begin{defn}\label{vancyc}
Define a vanishing cycles gluing data for $f$ to be a quadruple $(\mathcal{F}_V , \mathcal{F}_Z , u, v)$: $\mathcal{F}_V \in \text{Perv}(V),$ $ \mathcal{F}_Z \in \text{Perv}(Z)$ and maps ${\Psi}_f^{un}(\mathcal{F}_V) \xrightarrow{u} \mathcal{F}_Z \xrightarrow{v} {\Psi}_f^{un} (\mathcal{F}_V)(-1)$, where $v \circ u = N$.\\
Let $\text{Perv}_f(V, Z)$ be the category of gluing data with morphisms being maps between $h_{V}:\mathcal{F}_{V} \to \mathcal{G}_{V}$ and $h_{Z}:\mathcal{F}_{Z} \to \mathcal{G}_{Z}$ making the diagrams commute:

\[
\xymatrix{
{\Psi}_f^{un}(\mathcal{F}_V) \ar[r]^{\;\;u\;\;} \ar[d]_{\Psi_f^{un}(h_V)} & 
\mathcal{F}_Z \ar[r]^{\;\;v\;\;} \ar[d]_{h_Z} & 
{\Psi}_f^{un}(\mathcal{F}_V)(-1) \ar[d]_{\Psi_f^{un}(h_V)(-1)} \\
{\Psi}_f^{un}(\mathcal{G}_V) \ar[r]_{\;\;u\;\;} & 
\mathcal{G}_Z \ar[r]_{\;\;v\;\;} & 
{\Psi}_f^{un}(\mathcal{G}_V)(-1)
}
\]

In particular, for any $\mathcal{F} \in \text{Perv}(X)$, the quadruple $F_f (F) = (j^{\ast} \mathcal{F}, \bm{\tau}_f^{un}(\mathcal{F}),u,v)$ is such data and gives the functor $F_f : \text{Perv}(X) \rightarrow \text{Perv}_f(V, Z)$.  \end{defn}
Conversely, given a vanishing cycles data
\[ {\Psi}_f^{un}(\mathcal{F}_V) \xrightarrow{u} \mathcal{F}_Z \xrightarrow{v} {\Psi}_f^{un} (\mathcal{F}_V) (-1)\]
we can form the complex 
\[ {\Psi}_f^{un} (\mathcal{F}_V) \xrightarrow{(\beta_{+},u)} \Xi_f^{un}(\mathcal{F}_V) \oplus \mathcal{F}_Z \xrightarrow{(\beta_{-},-v)} {\Psi}_f^{un}(\mathcal{F}_V) (-1)\]
since $v \circ u = N = \beta_{-} \circ \beta_{+}$. 
Let $G_f(\mathcal{F}_V, \mathcal{F}_Z, u,v)$ be its cohomology sheaf.   
\begin{Theorem}(Theorem 3.6 in \cite{rr})\label{perveq}
The gluing category $\text{Perv}_f(V,Z)$ is abelian; $F_f: \text{Perv}(X) \to \text{Perv}_f(V,Z)$ and $G_f: \text{Perv}_f(V,Z) \to \text{Perv}(X)$ are mutually inverse exact functors, and so $\text{Perv}_f(V,Z)$ is equivalent to $\text{Perv}(X)$.     
\end{Theorem}

\subsection{Gluing for the Bi-Whittaker Category}
In this subsection, we will set up the analogous results for the bi-Whittaker category $e_{\mathcal{L}}\mathscrsfs{D}(G)e_{\mathcal{L}}.$
We first recall some results mentioned in \cite{ac}:
\begin{prop}\label{ff1} (Proposition A.4.16 in \cite{ac}):
Let $F : \mathcal{T} \to \mathcal{T}^{\prime}$ be a triangulated functor, and let
$\mathcal{S} \subset \mathcal{T}$ be a set of objects that generates $\mathcal{T}$. If the map
\[ \text{Hom}_{\mathcal{T}}(X,Y[n]) \to \text{Hom}_{\mathcal{T}^{\prime}}(F(X), F(Y[n]))\]
is an isomorphism for all $X,Y \in \mathcal{S}$, then $F$ is fully faithful.
   
\end{prop}

\begin{prop}\label{ff2}(Proposition A.4.17 in \cite{ac})
Let $F : \mathcal{T} \to \mathcal{T}^{\prime}$ be a fully faithful triangulated functor. If the image of $F$ contains a set of objects that generates $\mathcal{T}$, then $F$ is an equivalence of categories.
   
\end{prop}

We have that the following version of Beilinson's Theorem (in \cite{beil2}) also holds for our bi-Whittaker category:
\begin{lemma}\label{ff3}
If $\text{Perv}\left(e_{\mathcal{L}}\mathscrsfs{D}(G)e_{\mathcal{L}}\right)$ denotes the full subcategory of perverse objects in the bi-Whittaker category $e_{\mathcal{L}}\mathscrsfs{D}(G)e_{\mathcal{L}}$\footnote{So $\text{Perv}\left(e_{\mathcal{L}}\mathscrsfs{D}(G)e_{\mathcal{L}}\right)$ is a full subcategory of \text{Perv}(G)}, we have that Beilinson's realization functor:
\[ \text{real : } D^b\left( \text{Perv}\left(e_{\mathcal{L}}\mathscrsfs{D}(G)e_{\mathcal{L}}\right)\right) \to e_{\mathcal{L}}\mathscrsfs{D}(G)e_{\mathcal{L}} \] 
is an equivalence of categories.
\end{lemma}

\begin{proof}
Any object in $e_{\mathcal{L}}\mathscrsfs{D}(G)e_{\mathcal{L}}$ can be expressed as extensions of objects in $e_{\mathcal{L}}\mathscrsfs{D}(G)e_{\mathcal{L}}$ supported on the Bruhat cells. We have from the equivalence in \ref{dictionary}, that $\mathscrsfs{D}(N_{w,\mathcal{L}}) \xrightarrow[]{\cong}  e_{\mathcal{L}}\mathscrsfs{D}(BwB)e_{\mathcal{L}}$ and under the equivalence, the perverse subcategory goes to the perverse subcategory (upto a shift). Therefore, we reduce to the fact that objects in $\text{Perv}(N_{w,\mathcal{L}})$ generate the derived category $\mathscrsfs{D}(N_{w,\mathcal{L}})$ which is true (this follows from the Theorem of Beilinson that we are referring to). Therefore we have that $e_{\mathcal{L}}\mathscrsfs{D}(G)e_{\mathcal{L}}$ is generated by objects in $\text{Perv}\left(e_{\mathcal{L}}\mathscrsfs{D}(G)e_{\mathcal{L}}\right)$.\\

Finally, by the two propositions mentioned above, it is enough to show that for any two perverse sheaves, $\mathcal{F}, \mathcal{G} \in \text{Perv}\left(e_{\mathcal{L}}\mathscrsfs{D}(G)e_{\mathcal{L}}\right)$, and any $k \ge 0$, the map 
\[\text{real}: \text{Ext}^k_{\text{Perv}\left(e_{\mathcal{L}}\mathscrsfs{D}(G)e_{\mathcal{L}}\right)}(\mathcal{F}, \mathcal{G}) \to \text{Hom}_{e_{\mathcal{L}}\mathscrsfs{D}(G)e_{\mathcal{L}}}(\mathcal{F}, \mathcal{G}[k])\]
induced by real is an isomorphism. But, this again follows from the general case (Theorem 4.5.9 in \cite{ac}) for the enitre derived category $\mathscrsfs{D}(G)$, as we defined $e_{\mathcal{L}}\mathscrsfs{D}(G)e_{\mathcal{L}}$ to be a full subcategory.
\end{proof}
The following lemmas will allow us to say that the various functors, such as $\Xi^{un}_f, \bm{\tau}^{un}_f, {\Psi}^{un}_f$ (talked about in the previous subsection) preserve our bi-Whittaker category.\\
We have, following \cite{bdr}, that $\mathscrsfs{D}(G)$ is a Grothendieck-Verdier Category with the dualizing operation as $\mathbb{D}^{-}_{G}= \mathbb{D}_{G}\circ \iota^{\ast} = \iota^{\ast} \circ \mathbb{D}_G$, where $\mathbb{D}_G$ is the usual Verdier Duality functor and $\iota$ is the inverse map of the group $G$. 
\begin{lemma}
The bi-Whittaker $e_{\mathcal{L}}\mathscrsfs{D}(G)e_{\mathcal{L}}$ is preserved under the duality functor $\mathbb{D}^{-}_{G}$ described above.
\end{lemma}
\begin{proof}
Since $e_{\mathcal{L}}$ is a closed idempotent, this directly follows from Lemma A.50 of op. cit. 
\end{proof}
As an immediate consequence we get the following:
\begin{prop} \label{ext1}
If $\mathcal{F} \in e_{\mathcal{L}}\mathscrsfs{D}(G)e_{\mathcal{L}}$ with support on $S\subset G$, then for the inclusion morphism $j: S \into G$, we have that $j_{!}{j^{!} \mathcal{F}} $ and $j_{!}j^{\ast}\mathcal{F} $ also lie in the bi-Whittaker Category $e_{\mathcal{L}}\mathscrsfs{D}(G)e_{\mathcal{L}}$.    
\end{prop} 
\begin{proof}

Clearly, we have that $j_{!}\mathcal{F}$ lies in $e_{\mathcal{L}}\mathscrsfs{D}(G)e_{\mathcal{L}}$. We denote by $^{\iota}j\text{: } ^{\iota}S \into$ $G$, the inclusion being pulled back by the inverse map of the group.\\
First, we see that $j_{\ast} \mathcal{F} = \mathbb{D}_G(j_{!}(\mathbb{D}_G\mathcal{F})) = \mathbb{D}_G(j_{!}(\iota^{\ast}\circ\mathbb{D}^{-}_{G}(\mathcal{F}))) = \mathbb{D}_G(\iota^{\ast}(^{\iota}j_{!} (\mathbb{D}^{-}_G (\mathcal{F})))) = \mathbb{D}^{-}_G(^{\iota}j_{!} (\mathbb{D}^{-}_G (\mathcal{F}))).$ \\
Now, from the lemma above, we have that $\mathbb{D}_G^{-}$ preserves $e_{\mathcal{L}}\mathscrsfs{D}(G)e_{\mathcal{L}}$ and $^{\iota}j_{!} (\mathbb{D}^{-}_G (\mathcal{F}))$ will lie in the bi-Whittaker Category. Therefore, we get that $j_{\ast}\mathcal{F}$ also lies in the bi-Whittaker Category.  
\end{proof}
\begin{lemma}\label{ext2}
If $\mathcal{F} \in e_{\mathcal{L}}\mathscrsfs{D}(G) e_{\mathcal{L}}$ and is supported on a Bruhat cell $BwB$ (some $w\in W$), and $f$ a regular function on $\overline{BwB}$ that is only vanishing on $\overline{BwB} \backslash BwB$ (the kind discussed in Section \ref{Schubertsection}), we have that $\mathcal{F} \otimes f^{\ast}\mathcal{K}$ also lies in $e_{\mathcal{L}}\mathscrsfs{D}(G)e_{\mathcal{L}}$, where $\mathcal{K}$ is a local system on $\Gm$.
\end{lemma}

\begin{proof}
Consider the following square which we can get out of the discussion in Section \ref{Schubertsection}, $f$ is non-vanishing on $BwB$ and vanishes only on $\overline{BwB}\backslash BwB$.
\[\xymatrix{
 BwB \ar[r] \ar[d]^{f} & \overline{BwB} \ar[d] & \overline{BwB}\backslash BwB \ar[l] \ar[d] \\
 \Gm \ar[r] & \mathbb{A}^{1} & 0 \ar[l] .
}\] 
We first observe that the local system $f^{\ast} \mathcal{K}$ on $BwB$ is constant along all unipotent orbits: this follows because the unipotent orbits are all affine spaces which map trivially to $\Gm$.

Recall that $BwB = U \times Tw \times U_w$ and so, the constancy along the $U$ and $U_w$ orbits says the the local system $f^{\ast}\mathcal{K}$ on $BwB$ is also the pull back of $(f^{\ast}\mathcal{K})_{|Tw}$ along the projection to $BwB = U \times Tw \times U_w \to Tw$.\\
This, along with the equivalence that has been described earlier in \ref{dictionary}: $e_{\mathcal{L}}\mathscrsfs{D}(BwB)e_{\mathcal{L}} \xleftrightarrow{m,n} \mathscrsfs{D}(N_{w,\mathcal{L}})$ gives us that $f^{\ast}\mathcal{K} \otimes \mathcal{F} \cong m\left(n(\mathcal{F}) \otimes f^{\ast}(\mathcal{K})_{|Tw} \right)$, in particular, $f^{\ast}{\mathcal{K}}\otimes \mathcal{F}$ lies in $e_{\mathcal{L}} \mathscrsfs{D}(G) e_{\mathcal{L}}$.     
\end{proof}
\begin{lemma}\label{nearby} If $\mathcal{F} \in e_{\mathcal{L}}\mathscrsfs{D}(G)e_{\mathcal{L}}$ is a perverse sheaf supported on $BwB$ and $f$ is as above, then we have that ${\Psi}_f^{un}(\mathcal{F})$ is also lying in the bi-Whittaker category $e_{\mathcal{L}}\mathscrsfs{D}(G)e_{\mathcal{L}}$.
\end{lemma}

\begin{proof}
${\Psi}_f^{un}(\mathcal{F})$ has been defined as the stable kernel of the maps $\alpha_a : j_{!}(\mathcal{M} \otimes f^{\ast}\mathcal{K}^a) \to j_{\ast}(\mathcal{M} \otimes f^{\ast}\mathcal{K}^a)$. Following lemmas \ref{ext1} and \ref{ext2}, both the sheaves are perverse and are lying in $e_{\mathcal{L}}\mathscrsfs{D}(G)e_{\mathcal{L}}$.\\
Therefore, the kernel is also perverse and lying in $e_{\mathcal{L}}\mathscrsfs{D}(G)e_{\mathcal{L}}$ ($\text{Perv }(e_{\mathcal{L}}\mathscrsfs{D}(G)e_{\mathcal{L}})$ is a full abelian subcategory of $\text{Perv}(G)$ as discussed).   \end{proof}
Recall that the construction of the maximal extension functor uses the unipotent local system $\mathcal{K}$ as above, and just like ${\Psi}_f^{un}$, it is the stable kernel (and cokernel) of a family of morphisms between perverse sheaves lying in $e_{\mathcal{L}}\mathscrsfs{D}(G)e_{\mathcal{L}}$ and tensored with powers of $\mathcal{K}$. Therefore, we get exactly the same proof of the fact that for $\mathcal{F} \in e_{\mathcal{L}}\mathscrsfs{D}(G)e_{\mathcal{L}}$, the object $\Xi_f^{un}$ also lies in $e_{\mathcal{L}}\mathscrsfs{D}(G)e_{\mathcal{L}}$.

As a consequence, we get the following version of Theorem \ref{perveq}; but first we set up some new notation: From now let $\mathcal{H}(G):= \text{Perv}(e_{\mathcal{L}}\mathscrsfs{D}(G)e_{\mathcal{L}})$ denote the perverse sheaves in the bi-Whittaker Category $e_{\mathcal{L}}\mathscrsfs{D}(G)e_{\mathcal{L}}$ and similarly, for a subscheme $X\subseteq G$ that is $B$-bi-invariant, let $\mathcal{H}(X)$ denote the full subcategory of $\mathcal{H}(G)$ with support lying in $X$. Specifically, we consider $X$ which is a union of closures of same dimensional Bruhat cells and let $V$ be the union of the open Bruhat cells with $Z$ the closed complement. We define the category $\mathcal{H}_{f}(V,Z)$ analogously, following \ref{vancyc} as the category of vanishing cycles data for the function $f$ and the category $\mathcal{H}(X).$

\begin{Cor}\label{perveqhecke} The gluing category $\mathcal{H}_f(V,Z)$ is abelian; restricting to the bi-Whittaker Category $e_{\mathcal{L}}\mathscrsfs{D}(G)e_{\mathcal{L}}$, $F_f: \mathcal{H}(X) \to \mathcal{H}_f(V,Z)$ and $G_f: \Hcal_f(V,Z) \to \Hcal(X)$ are mutually inverse exact functors, and so $\Hcal_f(V,Z)$ is equivalent to $\Hcal(X)$.     
\end{Cor}
We can now use the above result for our purposes as follows: We have the antiautomorphism $\bm{\Psi} : G \to G$ from Section \ref{section antihom}, and as discussed at the beginning of this section (\ref{bruhatnt}) we have that $\bm{\Psi}^{\ast} \mathcal{F} \cong \mathcal{F}$, functorially, for objects $\mathcal{F}$ in $e_{\mathcal{L}}\mathscrsfs{D}(G)e_{\mathcal{L}}$ and supported in a single Bruhat cell. We will now use Beilinson's gluing to upgrade this to a natural transformation between the functor $\bm{\Psi}^{\ast}$ and the identity functor on $e_{\mathcal{L}}\mathscrsfs{D}(G)e_{\mathcal{L}}$.

\begin{Theorem}\label{nt} The functor $\bm{\Psi}^{\ast}: 
e_{\mathcal{L}}\mathscrsfs{D}(G)e_{\mathcal{L}} \to e_{\mathcal{L}}\mathscrsfs{D}(G)e_{\mathcal{L}}$
admits a natural isomorphism $\theta: \bm{\Psi}^{\ast} \to \text{Id}$ (to the identity functor on the category $e_{\mathcal{L}}\mathscrsfs{D}(G)e_{\mathcal{L}}$).

\end{Theorem}

\begin{proof}
From \ref{ff3} we have that the bi-Whittaker  category is the derived category of its perverse heart and so, it is enough to construct a natural transformation for perverse sheaves first.

We have already set it up for sheaves supported on a single Bruhat cell so in particular we also have it set up for such perverse sheaves (refer \ref{bruhatnt}). \\
Therefore, to prove it for general perverse sheaves, we use induction on the dimension of the maximal dimensional Bruhat cell that intersects the support of a given sheaf; it may be the case that there are multiple such Bruhat cells intersecting the support. Our main idea to construct the natural transformation is to exploit the equivalence of categories set up in \ref{perveqhecke}. From that corollary, it is enough to see that $\bm{\Psi}^{\ast}$ is inducing a natural transformation to the identity functor on the categories $\Hcal_{f}(V,Z)$ (the open set $V$ possibly being the union of multiple Bruhat cells of the same dimension and $Z = \overline{V}\backslash V$ the closed complement.  The equivalence is established by the functor $F_{f}$ from \ref{perveqhecke}).

In \ref{functions} we establish the existence of such separating functions $f$ even in the case where there are multiple Bruhat cells of the same dimension. 
The objects in the category $\Hcal_{f}(V,Z)$ are vanishing cycles data as has been described: we need objects $\mathcal{F}_V$ and $\mathcal{F}_Z$ along with the morphisms $u$ and $v$ such that $v \circ u = N$. In the diagram below, we have already established for $\mathcal{F}_V \in \Hcal(V)$ and $\mathcal{F}_Z \in \Hcal(Z)$ that the downward pointing arrows are isomorphisms (induction hypothesis). We also use the natural isomorphism $\bm{\Psi}^{\ast}({\Psi}_f^{un}(\mathcal{F}_V)) \cong {\Psi}_f^{un}(\bm\Psi^{\ast}(\mathcal{F}_V))$. 
\[\xymatrixcolsep{4pc}\xymatrix{
 {\Psi}_f^{un}(\mathcal{F}_V) \ar[r]^-{u} \ar[d]^{\cong} & \mathcal{F}_Z\ar[r]^-{v}\ar[d]^{\cong} & {\Psi}_f^{un}(\mathcal{F}_V)(-1)\ar[d]^{\cong} \\
{\Psi}_f^{un} (\bm{\Psi}^{\ast}(\mathcal{F}_V))\ar[r]^-{\bm{\Psi}^{\ast}u} & \bm{\Psi}^{\ast}(\mathcal{F}_Z) \ar[r]^-{\bm{\Psi}^{\ast} v} &{\text{ 
       }}{\Psi}_f^{un} (\bm{\Psi}^{\ast}(\mathcal{F}_V))(-1)
}\] 

We now only need to show that the squares above commute but that again follows from induction hypothesis: that the isomorphisms we are establishing are giving us a natural transformation between $\bm{\Psi}^{\ast}$ and the $\text{Id}$ functor (the sheaves in the diagram above have support of lesser dimension). The base case is realized by the isomorphisms being identity when restricted on the torus $T$.

From applying \ref{r1} to all the possible collections of various choices of $f$'s to perform the multiple gluing steps, it is again clear(via induction) that the natural isomorphism does not depend on the choice of the functions. Let us call this natural transformation $\theta: \bm{\Psi}^{\ast} \mathcal{F} \to \mathcal{F}$.
 
\end{proof}

\section{Equivalence of Symmetric Monoidal Structures}\label{symmeq}

We start with briefly recalling the construction of the monoidal functor $ \zeta: e_{\mathcal {L}} \mathscrsfs{D}(G) e_{\mathcal{L}} \xrightarrow[]{\cong}
 \mathscrsfs{D}_{W}^{\circ}(T)$ constructed in \cite{bd} ($\Dcirc_W(T) \subseteq \D_W(T)$ is the full monoidal subcategory of central sheaves, refer to \ref{central}, also look at Section \ref{Schubertsection} of loc. cit. for more details). 

\begin{defn}
The Yokonuma-Hecke category is defined to be the triangulated monoidal category 
\[\mathscrsfs{D}(U \backslash G/U ) = e_U \mathscrsfs{D}(G)e_U \subseteq \mathscrsfs{D}_U(G) \subseteq \mathscrsfs{D}(G)\] 
whose unit object is $e_U$.\\
Here, following \cite{bd}, we will instead be looking at the Yokonuma-Hecke category $e_{U^{-}}\mathscrsfs{D}(G)e_{U^{-}}$, with respect to the opposite unipotent subgroup $U^{-}$, whose
unit object is $e_{U^{-}}$.     
 \end{defn}
The main idea in loc. cit. is to use that the full subcategory $e_{U^{-}}\mathscrsfs{D}(G)e_{\mathcal{L}}\subset \D(G)$ is a left $e_{U^{-}}\mathscrsfs{D}(G)e_{U^{-}}$-module category and a right $e_{\mathcal{L}} \mathscrsfs{D}(G) e_{\mathcal{L}}$-module category and these two actions are compatible.\\
The open subvariety $U^{-}TU \subseteq G$ is isomorphic to the product $U^{-} \times T \times U$. One checks that the support of objects in $e_{U^{-}}\mathscrsfs{D}(G)e_{\mathcal{L}}$ must lie in $U^{-}TU$, we refer to \cite{bd}, lemma 1.5 for this. It also gives us a triangulated equivalence $\mathscrsfs{D}(T) \xrightarrow[]{\cong} e_{U^{-}}\mathscrsfs{D}(U^{-}TU)e_{\mathcal{L}} = e_{U^{-}}\mathscrsfs{D}(G)e_{\mathcal{L}}$ given for $C \in \mathscrsfs{D}(T)$ as

\[ C \mapsto e_{U^{-}} \ast C \ast e_{\mathcal{L}}.\]

The inverse functor is given by 
\[ \mathcal{F} \in e_{U^{-}} \mathscrsfs{D}(G)e_{\mathcal{L}} \mapsto \mathcal{F}_{|T}[-4\text{dim} U](-2\text{dim} U)\]

Following this, a certain quotient $\mathscrsfs{Y}$ of the Yokonuma-Hecke category modulo a thick triangulated two-sided ideal $\mathscrsfs{I} \subseteq e_{U^{-}}\mathscrsfs{D}(G)e_{U^{-}}$ is described. The ideal $\mathscrsfs{I}$ is such that its left action on the bimodule category $e_{U^{-}}\mathscrsfs{D}(G)e_{\mathcal{L}} \cong \mathscrsfs{D}(T)$ is trivial and hence $\mathscrsfs{D}(T)$ gets the structure of a $\mathscrsfs{Y}-e_{\mathcal{L}}\mathscrsfs{D}(G) e_{\mathcal{L}}$-bimodule category, for $\mathscrsfs{Y}:= e_{U^{-}}\mathscrsfs{D}(G)e_{U^{-}}/ \mathscrsfs{I}.$

For each simple root $\alpha$, the authors use the objects $e_{s_{\alpha}}$ and $\mathcal{K}_{s_{\alpha}}$, which are the Kazhdan-Laumon sheaves (\cite{kl}). 

For $w \in W$, where $w = s_1 \cdots s_l$ is a reduced expression as a product of simple expressions, $\mathcal{K}_w:= \mathcal{K}_{s_1} \ast \cdots \ast \mathcal{K}_{s_l} \in e_{U^{-}}\mathscrsfs{D}(G)e_{U^{-}}$ is independent of the choice of reduced expression.
It is also shown that the functor $\underline{W} \to \mathscrsfs{Y}$ defined by $w \mapsto \mathcal{K}_w$ (viewed as an object in the quotient modulo $\mathscrsfs{I}$) has a natural monoidal structure, and additionally, the left action of $\mathcal{K}_w$ on the bimodule category $\mathscrsfs{D}(T)$ coincides with the adjoint action of $w$ on $\mathscrsfs{D}(T)$.

As the endofuctor induced by an object of $e_{\mathcal{L}} \mathscrsfs{D}(G) e_{\mathcal{L}}$ on $e_{U^{-}}\mathscrsfs{D}(G)e_{\mathcal{L}} \cong \mathscrsfs{D}(T)$ gives rise to a left $\mathscrsfs{Y}$ module action, and the fact that we have a monoidal functor 
$\mathscrsfs{D}(T) \cong e_{U^{-}}\mathscrsfs{D}(B^{-}) \subseteq e_{U^{-}} \mathscrsfs{D}(G) e_{U^{-}}  \to \mathscrsfs{Y}$ with $\mathscrsfs{D}(T)$ acting as usual convolution, it follows that to determine the endofunctor, it is enough to know the image of the object $\delta_{1}$.\\
Finally, the additional data of the isomorphisms using the functor $\underline{W}\to \mathscrsfs{Y}$ determines the functor  \newline $ \zeta: e_{\mathcal {L}} \mathscrsfs{D}(G) e_{\mathcal{L}} \xrightarrow[]{\cong}
 \mathscrsfs{D}_{W}^{\circ}(T)$ ($\mathscrsfs{D}_{W}^{\circ}(T)$ is a full subcategory of $\mathscrsfs{D}_{W}(T)$ to which the functor is a priori defined, refer \ref{central} in this paper and \cite{bd}).

 In summary, we have that under this functor, an object $A \in e_{\mathcal{L}}\mathscrsfs{D}(G)e_{\mathcal{L}}$ maps to the pair

 $\left((e_{U^{-}} \ast A)_{|T}[-4 \text{dim} U] (-2 \text{dim} U), \rho ^{A}\right)$ 
where $\rho^{A}$ is the $W$-equivariant structure coming from the left $\mathscrsfs{Y}$-module endofunctor structure of the right action of $A$ on the bimodule category $\mathscrsfs{D}(T)$.

For the sake of completion, we recall the precise statement proved in \cite{bd}

\begin{Theorem} (Theorem 1.4 in \cite{bd})\label{zetathm}
There is a triangulated monoidal equivalence.
\[  \zeta: e_{\mathcal {L}} \mathscrsfs{D}(G) e_{\mathcal{L}} \xrightarrow[]{\cong}
 \mathscrsfs{D}_{W}^{\circ}(T) \]
 whose inverse is given by the composition 
 \[\mathscrsfs{D}^{\circ}_W(T) \xrightarrow{\text{ind}^W}\mathscrsfs{D}_G^{\circ}(G)\xrightarrow{\text{HC}_{\mathcal{L}}} e_{\mathcal {L}} \mathscrsfs{D}(G) e_{\mathcal{L}}\]
In particular the bi-Whittaker category $e_{\mathcal {L}} \mathscrsfs{D}(G) e_{\mathcal{L}}$ has the structure of a triangulated symmetric monoidal
category. Moreover, the equivalence above is a $t$-exact equivalence for the perverse $t$-structure on $\mathscrsfs{D}^{\circ}_W(T)$ and
the perverse $t$-structure shifted by $\text{dim }U$ on $e_{\mathcal {L}} \mathscrsfs{D}(G) e_{\mathcal{L}}$. 
 
\end{Theorem}

There are now two things that remain to be shown: firstly that the symmetric monoidal structure as proposed in Remark \ref{flip} indeed satisfies the associativity coherence condition and secondly that this symmetric monoidal structure agrees with the one constructed in \cite{bd}. We will show, through our comparison that the associativity coherence for our structure comes for free.\\
We begin with the preparatory step:
\begin{lemma}\label{zero}
Let $X$ be a scheme of finite type over $k=\overline{k}$. Say we have an endomorphism of the identity functor $\tilde{\pi}$ on $\mathscrsfs{D}(X)$ which has been induced by an endomorphism $\pi: \text{Id} \to \text{Id}$ of the identity functor on the subcategory of perverse sheaves $\text{Perv}(X)$ of $\mathscrsfs{D}(X)$. If $\pi$ acts as the zero endomorphism on all skyscraper sheaves, then the endomorphisms $\pi$ and hence $\tilde{\pi}$ are infact, zero as endomorphisms of the respective identity functors. 
\end{lemma}

\begin{proof}

Let $\mathcal{F}$ be a local system on a locally closed $U \subset X$. For the inclusion of a closed point $i: \{x\} \into X$, we have the adjunction morphism $i^{\#}: \mathcal{F}\to i_{\ast}i^{\ast} \mathcal{F}$ and the corresponding natural transformation square:
\[\xymatrix{
 \mathcal{F}\ar[r]^{\tilde{\pi}} \ar[d]^{i^{\#}} & \mathcal{F} \ar[d]^{i^{\#}}\\
 i_{\ast}i^{\ast} \mathcal{F}\ar[r]^{\tilde{\pi}} & i_{\ast}i^{\ast} \mathcal{F}.
}\] 
As the fiberwise maps are zero for all $x$, the local system $\mathcal{F}$ is sent to zero.\\
As perverse sheaves satisfy smooth descent, we can assume that $X$ is affine. Let's consider a smooth locally closed affine subset $U \subset X$. For a local system $\mathcal{L}$ on $U$ we have $\mathcal{L}[\text{dim } U] \in \text{Perv}(U)$. From the above considerations, it follows that $\pi$ is in fact zero on the objects $\mathcal{L}[\text{dim } U]$ on $\text{Perv}(U)$ (this category embeds into $\Perv(X)$ via the $j_{!}$ functor). \\
Now, given an arbitrary $\mathcal{F} \in \text{Perv}(X)$ there is a locally closed smooth affine $U \xhookrightarrow{j} \overline{\text{supp}\mathcal{F}} \to X$ such that $j^{\ast}\mathcal{F}$ is isomorphic to $\mathcal{L}[\text{dim }U]$ for some local system $\mathcal{L}$ on $U$, with $U$ open in $\overline{\text{supp}(\mathcal{F})}$ and the closed complement $\overline{\text{supp}(\mathcal{F})}\textbackslash U$ of lesser top dimension than $U$ and can be written as the vanishing locus of some function $f \in \mathcal{O}(\overline{\text{supp}(\mathcal{F})})$. This again sets up Beilinson's machine of gluing of perverse sheaves. \\
So now following \ref{perveq}, we induct on the top dimension of the support of the perverse sheaves with the claim being that the endomorphism is zero. Then by induction hypothesis, the endomorphism of the vanishing cycles on the closed complement is zero and by the considerations above, the endomorphism on $j^{\ast}\mathcal{F}$ is zero well, and the uniqueness of the gluing construction gives us that the endomorphism of the perverse sheaf that we started with is zero.
\end{proof}
\begin{Cor}\label{const}
   Let $X$ be a connected scheme of finite type over $k=\overline{k}$. Say we have an endomorphism of the identity functor $\tilde{\pi}$ on $\mathscrsfs{D}(X)$ which has been induced by an endomorphism $\pi: \text{Id} \to \text{Id}$ of the identity functor on the subcategory of perverse sheaves $\text{Perv}(X)$ of $\mathscrsfs{D}(X)$. Then $\pi$ is in fact given by multiplication by a scalar $c\in \Q_{\ell}$.
\end{Cor}
\begin{proof}
    Firstly, since $X$ is connected, the endomorphism algebra of the constant sheaf $ \overline{\Q}_{\ell}$ is just $\overline{\Q}_{\ell}$. Say the natural transformation $\pi$ acts on the constant sheaf by the scalar $c\in \overline{\Q}_{\ell}$. Now, for a closed point $x \in X$, the inclusion being given by $i: \{x\} \into X$, we draw the natural transformation square for the map $i^{\#}:\Qlcl \to i_{\ast}i^{\ast} \Qlcl$, coming from adjunction.

\[\xymatrix{
 \Qlcl\ar[r]^{\tilde{\pi}} \ar[d]^{i^{\#}} & \Qlcl \ar[d]^{i^{\#}}\\
 i_{\ast}i^{\ast} \Qlcl\ar[r]^{\tilde{\pi}} & i_{\ast}i^{\ast} \Qlcl.
}\] 

As a consequence, we can see that the natural transformation must act on every skyscraper sheaf by $c$. \\
We now apply the previous proposition to $\pi - c\cdot\text{Id}$ to conclude.

\end{proof}

\begin{rmk}\label{scalardefn}
Given $w\in W$, consider $\pi_0({N_{w,\mathcal{L}}})$, the set of connected components of $N_{w,\mathcal{L}}$ (as defined in \ref{nwldef}). Given a natural transformation $\Psi^{\ast} \to \text{Id}$ (which is induced from the perverse subcategory), for every $\beta \in \pi_0({N_{w,\mathcal{L}}})$ we have a constant $c_{w,\beta}$ that describes the natural transformation when restricted to $UN_{\beta, w, \mathcal{L}} U$, via being multiplication by $c_{w,\beta}$ on $\mathscrsfs{D}(N_{\beta, w, \mathcal{L}})$ under the equivalence in \ref{dictionary}. This follows from \ref{const} above, where $N_{\beta, w, \mathcal{L}} \subseteq N_{w,\mathcal{L}}$ is the connected component that $\beta$ corresponds to.
\end{rmk}

\begin{rmk}\label{natmaps}
We have that the category $\mathscrsfs{D}_W(T)$ comes along with some naturally defined functors, each of which admits a natural transformation to the identity functor: \\
Recall that the data of an object in $\mathscrsfs{D}_W(T)$ is
\begin{itemize}
    \item An $\mathcal{F} \in \mathscrsfs{D}(T)$
    \item And isomorphisms $\alpha_w: ^w\mathcal{F} \to \mathcal{F}$ along with the cocycle conditions $\alpha_v \circ {^v\alpha_w} = \alpha_{vw}$.
\end{itemize}
For each $w\in W$, we define functors $s_w:\mathscrsfs{D}_W(T) \to \mathscrsfs{D}_W(T)$  
\begin{itemize}
    \item Sending $^w\mathcal{F}$ to $\mathcal{F}$.
    \item The transition morphisms for $^w\mathcal{F}$ being the transition morphisms of $\mathcal{F}$ being shifted by $w$.
\end{itemize} 
The starting cocycle conditions for $\mathcal{F}$ verify for us that this is indeed a functor.
There is also a natural transformation $s_w \to \text{Id}$ where, for an object $\mathcal{F}$, it is given by the morphism \\ $w^{-1}\alpha_w: \mathcal{F} \to {^{w^{-1}}}\mathcal{F}$. We shall abuse notation and call this natural transformation $\alpha_w$ as well.
\end{rmk}
  
\begin{rmk}\label{r2}
As $s_{w_0}$ is an anti-monoidal functor on the torus ($T$ is commutative, so monoidal functors are also anti-monoidal), we can, like we do in \ref{flip} for $e_\mathcal{L}\mathscrsfs{D}(G)e_{\mathcal{L}}$, use the natural transformation $\alpha_{w_0}$ to propose a symmetric monoidal structure on $\mathscrsfs{D}_W(T)$. In fact, it is easy to see that by doing this, on $\mathscrsfs{D}_W(T)$, we will get the canonical symmetric monoidal structure obtained from the abelianness of $T$.\\
Now, continuing from Remark \ref{flip}, it remains to check the conditions for a symmetric monoidal structure for the one proposed for $e_{\mathcal{L}}\mathscrsfs{D}(G)e_{\mathcal{L}}$ using the natural transformation $\theta: \bm{\Psi}^{\ast} \to \text{Id}$. We will show that the natural transformation $\theta$ is exactly the one induced by $\alpha_{w_0}: s_{w_0} \to \text{Id}$ under $\zeta^{-1}$ and therefore the conditions will get verified as $\zeta$ is a monoidal equivalence and we indeed get a symmetric monoidal structure for $\mathscrsfs{D}_W^{\circ}(T)$ through the functor $s_{w_0}$ and the natural transformation $\alpha_{w_0}: s_{w_0} \to \text{Id}$.\\
We remark that this also shows the unicity of the symmetric monoidal structure on $e_{\mathcal{L}}\mathscrsfs{D}(G)e_{\mathcal{L}}$ with respect to the two different conditions.      
\end{rmk}

\begin{prop}\label{r1} 
Let  $U \xhookrightarrow{j} X \xhookleftarrow{i}Z$ be an open closed decomposition of a finite type scheme $X$ over $k$ realised via the zero and non-zero locus of a function $f\in \mathcal{O}(X)$. Let $\mathcal{P}(U), \mathcal{P}(X), \mathcal{P}(Z)$ be some full subcategories of $\text{Perv}(U), \text{Perv}(X), \text{Perv}(Z)$ respectively such that they are closed under the functors $j_{!},j_{\ast},j^{\ast}, i_{\ast},{\tau}^{un}_{f}, \Psi^{un}_{f}$ and let $\mathcal{N}(Z)$ be the induced subcategory of $\mathcal{P}(Z)$ under the nearby cycles functor on $\mathcal{P}(U)$. Let's consider two natural transformations $\eta_1, \eta_2$ of the identity functors on $\mathcal{P}(U)$ and $\mathcal{P}(Z)$. If the induced endomorphism on the subcategory $\mathcal{N}(Z)$ of $\mathcal{P}(Z)$ (from $\eta_1$) under the nearby cycles functor agrees with the restriction of $\eta_2$ on $\mathcal{P}(Z)$, then there is a unique endomorphism $\eta$ of the identity functor on $\mathcal{P}(X)$ extending $\eta_1$ and $\eta_2$.\\

\end{prop}

\begin{proof}
This is a straightforward consequence of Belinson's gluing setup mentioned in \ref{perveq}.    
\end{proof}
\begin{Cor}\label{constp}
    In the above setup, if there is an object $\mathcal{P} \in \mathcal{P}(X)$ which is a non-trivial extension, i.e., it is not of the form $j_{!}\mathcal{P}_1\oplus i_{\ast}\mathcal{P}_2$, for $\mathcal{P}_1,\mathcal{P}_2 \in \mathcal{P}(U)$ and $\mathcal{P}(Z)$ respectively. And, if $\text{End(Id}(\mathcal{P}_{U})) = \text{End(Id}(\mathcal{P}_{Z})) = \overline{\Q}_{\ell},$ then $\text{End(Id}(\mathcal{P}_{X})) = \overline{\Q}_{\ell}
$ and $\eta_1,\eta_2$ and $\eta$ are all given by multiplication by some $c\in \overline{\Q}_{\ell}$ on their identity functors. 
\end{Cor}
\begin{proof}
Let $c_1,c_2 \in \overline{\Q}_{\ell}$ be such that the natural transformations $\eta_1$ and $\eta_2$ are multiplication by $c_1$ and $c_2$ on $\mathcal{P}(U)$ and $\mathcal{P}(Z)$ respectively.\\
The condition on $\mathcal{P}$ exactly translates to $\bm{\Psi}^{un}_f(j^{\ast}\mathcal{P})\neq 0$. This is because, following Beilinson's construction (\ref{vancyc}, \ref{perveq}), if $\bm{\Psi}^{un}_f(j^{\ast}\mathcal{P}) = 0$, then all extensions of the Perverse sheaf $j^{\ast}{\mathcal{P}}$ to Perverse sheaves on $\text{Perv}(X)$ have to be trivial. The functoriality of the nearby cycles construction gives us that $\eta_2$ acts by multiplication by $c_1$ on $\bm{\Psi}^{un}_f(j^{\ast}\mathcal{P})$, and hence, $c_1 = c_2 $. By the uniqueness established in the previous proposition, the extension $\eta$ is unique and is given by multiplication by $c= c_1 = c_2$.

\end{proof}
\begin{rmk}\label{conarg}

 Any natural transformation between the functor $\bm{\Psi}^{\ast}: e_{\mathcal{L}}\mathscrsfs{D}(G)e_{\mathcal{L}} \to e_{\mathcal{L}}\mathscrsfs{D}(G)e_{\mathcal{L}}$ and $\text{Id}: e_{\mathcal{L}}\mathscrsfs{D}(G)e_{\mathcal{L}} \to e_{\mathcal{L}}\mathscrsfs{D}(G)e_{\mathcal{L}}$ (the identity functor), must restrict to a natural transformation for the full subcategory of objects in $e_{\mathcal{L}}\mathscrsfs{D}(G)e_{\mathcal{L}}$ supported on the Bruhat cell $BwB$ for every $w \in W$. Conversely, given such natural transformations for every $w$, we can attempt to glue them up to get a natural transformation on the whole bi-Whittaker category; this is what we do for instance in \ref{nt}.

In all the instances in this paper, the natural transformations of the Identity functor on a derived category will have been induced from the Perverse subcategory (for example $\theta: \bm{\Psi}^{\ast} \to \text{Id}$ on $e_{\mathcal{L}}\mathscrsfs{D}(G)e_{\mathcal{L}}$ in \ref{nt} and $s_{w_0}:\alpha_{w_0} \to \text{Id}$ on $\mathscrsfs{D}_W^{0}(T)$  in \ref{r2}). 

However, due to the equivalence $e_{\mathcal{L}}\mathscrsfs{D}(BwB)e_{\mathcal{L}} \xleftrightarrow{m,n} \mathscrsfs{D}(N_{w,\mathcal{L}})$ set up in \ref{dictionary} , any such natural transformation on $e_{\mathcal{L}}\mathscrsfs{D}(BwB)e_{\mathcal{L}}$ will induce a natural transformation on $\mathscrsfs{D}(N_{w,\mathcal{L}})$, which, due to Corollary \ref{const}, is just a choice of constant for every connected component. In particular, our construction in the Theorem \ref{nt}, has all the constants equal to $1$.

Following \ref{constp} however, it should not be possible for any choice of such constants on the Bruhat cells to glue to a global natural transformation. This is because of the presence of possible non-trivial extensions, i.e., sheaves on $\overline{BwB}$ which are non-trivial extensions of sheaves on $BwB$ and $\overline{BwB}\backslash BwB$ (a non-trivial extension in the derived category also gives a non-trivial extension in the perverse subcategory). Again following Corollary \ref{constp}, it is also clear that a natural transformation is uniquely determined from the constants on the Bruhat cells, so showing that the natural transformation induced from the one on $\mathscrsfs{D}_W^{\circ}(T)$ has $c_{w,\beta}=1$ for all pairs $(w,\beta)$ in $e_{\mathcal{L}}\mathscrsfs{D}(G)e_{\mathcal{L}}$ is enough to establish equality with the natural transformation that we have in our Theorem \ref{nt}.

 \end{rmk}

So reiterating, roughly our goals now, as alluded to in Remark \ref{conarg} are as follows: Show that the natural transformation between functors $\bm{\Psi}^{\ast}$ and $\text{Id}$ on $e_{\mathcal{L}}\mathscrsfs{D}(G)e_{\mathcal{L}}$ induced from the natural transformation $\alpha_{w_0}: s_{w_0} \to \text{Id}$ on $\mathscrsfs{D}_W^{\circ}(T)$ has $c_{w,\beta}=1$ for all pairs $(w,\beta)$. This will show the equality with the transformation $\theta$ that we have constructed. Before proving this, one of course needs to show that the image of the functor $\bm{\Psi}^{\ast}: e_{\mathcal{L}}\mathscrsfs{D}(G)e_{\mathcal{L}} \to e_{\mathcal{L}}\mathscrsfs{D}(G)e_{\mathcal{L}}$ under $\zeta$ is $s_{w_0}: \mathscrsfs{D}_{W}^{\circ}(T) \to \mathscrsfs{D}_{W}^{\circ}(T)$.\\
\begin{defn}(Scheme of tame characters.)
Let $\pi_1(T)^t$ be the tame \'etale fundamental group. A tame character is a continuous character $\chi: \pi_1(T)^t \to \overline{\Q_{\ell}}^{\times}$. In \cite{GL}, a $\overline{\Q_{\ell}}$-scheme $\mathcal{C}(T)$ is defined, whose $\overline{\Q_{\ell}}$-points are in bijection with the set of these tame characters. There is decomposition:
\[ \mathcal{C}(T) = \underset{\chi_f \in \mathcal{C}(T)_f}{\bigsqcup} \{ \chi_f\} \times \mathcal{C}(T)_{\ell} \]
into connected components, where $\mathcal{C}(T)_f \subset \mathcal{C}(T)$ is the subset consisting of tame characters of order prime to $\ell$ and $\mathcal{C}(T)_{\ell}$ is the connected component of $\mathcal{C}(T)$ containing the trivial character.
    
\end{defn}

\begin{defn} The Weyl group $W$ acts naturally on $\mathcal{C}(T)$ and for any $\chi \in \mathcal{C}(T)$, we denote
by $W_{\chi}$ the stabilizer of $\chi$ in $W$ and $W_{\chi}^{\circ}\subset W_{\chi}$ , the subgroup of $W_{\chi}$ generated by those reflections
$s_{\alpha}$ such that the pull-back $\check{\alpha}^{\ast} \mathcal{L}_{\chi}$ is isomorphic to the trivial local system, where $\check{\alpha} : \Gm \to T$ is the coroot associated to $\alpha$. \footnote{The notation above is following \cite{bd}, originally in \cite{chen1} and \cite{chen2}, $W_{\chi}\subseteq W_{\chi}^{\prime}$ has been used instead.}
\end{defn}

\begin{defn}(Central Sheaves)\label{central}
For any $\mathcal{F} \in  \mathscrsfs{D}_W(T)$ and $\chi \in \mathcal{C}(T)$, the $W-$equivariant structure on $\mathcal{F}$ together with the natural $W_{\chi}$-equivariant structure on $\mathcal{L}_{\chi}$ (local system associated to $\chi$) give rise to an action of $W_{\chi}$ on the cohomology groups $H^{\ast}_c(T, \mathcal{F}\otimes \mathcal{L}_{\chi})$. In \cite{bd}, the authors define the category $\mathscrsfs{D}^{\circ}_W(T)$  as the full subcategory of $\mathscrsfs{D}_W(T)$ spanned by objects $A\in \mathscrsfs{D}_W(T)$ such that the action of $W_{\chi}^{\circ}$ is trivial. \footnote{This differs from the one given in \cite{chen1}, \cite{chen2} by a sign character.}
\end{defn}

We now recall some properties of the {\em Mellin transform} constructed by the authors in \cite{GL}
\begin{equation}\label{mellindef}
\mathcal{M}_{!}: \mathscrsfs{D}(T) \to D^b_{coh}(\mathcal{C}(T))
\end{equation}
which is a monoidal functor with respect to the compact support convolution structure on $\mathscrsfs{D}(T)$ and the tensor product structure on $ D^b_{coh}(\mathcal{C}(T))$,
with the following relevant properties:
\begin{itemize}
    \item Let $\chi \in \mathcal{C}(T)(\overline{\Q_{\ell}})$ and $i_{\chi}: \{\chi\} \to \mathcal{C}(T)$ be the natural inclusion. We have
    \[ R\Gamma_c(T, \mathcal{F} \otimes \mathcal{L}_{\chi}) \cong i_{\chi}^{\ast}\mathcal{M}_{!}(\mathcal{F}) \]

    \item The Mellin transform restricts to an equivalence:
    \[ \mathcal{M}_{!}: \mathscrsfs{D}(T)_{\text{mon}} \cong D^b_{coh}(\mathcal{C}(T))_f \]
between the full subcategory $\mathscrsfs{D}(T)_{mon}$ of monodromic $\ell$-adic complexes on $T$ and the full subcategory $D^b_{coh}(\mathcal{C}(T))_f$ of coherent complexes on $\mathcal{C}(T)$ with finite support.

\end{itemize}
We now recall the following equivalent version of Proposition 4.2 of \cite{chen2}, suited for our definition of a central complex which differs from \cite{chen1} and \cite{chen2} by a sign character.
\begin{prop}\label{equivprop}
Let $\mathcal{F}\in \mathscrsfs{D}_W(T)$ be a $W$-equivariant complex. The following are equivalent:
\begin{altenumerate}
    \item $\mathcal{F}$ is central, as defined in \ref{central}.
    \item For any $\chi$ in $\mathcal{C}(T)(\overline{\Q_{\ell}})$, the action of $W_{\chi}^{\circ}$ on $i_{\chi}^{\ast}\mathcal{M}_{!}(\mathcal{F})$ is trivial.\\
    In the case $W_{\chi}^{\circ} = W_{\chi}$ for all $\chi \in \mathcal{C}(T)$, then the statements above are equivalent to 
    \item The restriction of the Mellin transfrom $\mathcal{M}_{!} (\mathcal{F}) \in D^b_{\text{coh}}(\mathcal{C}(T))$ to each connected component $\mathcal{C}(T)_{\ell,\chi_f}:= \{\chi_f \} \times \mathcal{C}(T)_{\ell}$ of $\mathcal{C}(T)$ descends to the quotient $\mathcal{C}(T)_{\ell,\chi_f}\backslash\backslash W_{\chi_f}$.
\end{altenumerate}
\end{prop}  

We also recall some facts about some specific tame local systems on the torus $T$ as described in \cite{chen1} and \cite{chen2}.
Let $\pi_1(T)_{\ell}$ be the pro-$\ell$ quotient of the tame fundamental group $\pi^t_1(T)$. \newline Let $R_T = \text{Sym}(\pi_1 (T)_{\ell} \otimes_{\Z_{\ell}} \overline{\Q_{\ell}})$ be the symmetric algebra of $\pi_1 (T)_{\ell} \otimes_{\Z_{\ell}} \overline{\Q_{\ell}}$ with $R_{T,+}$ being the argumentation ideal. Define $R_{\chi} = R_T/\langle R_{T,+}^{W_{\chi}^{\circ}} \rangle$ where $\langle R_{T,+}^{W_{\chi}^{\circ}} \rangle$ is the the ideal generated by $R_{T,+}^{W_{\chi}^{\circ}}$, the $W_{\chi}^{\circ}$ invariance of $R_{T,+}$ ($W_{\chi}$ has an action on $R_{T,+}$ through the action of $W$).

Consider the representation $R_{\chi}$ of $\pi_1(T)_{\ell}$, on which an element $\gamma \in \pi_1(T)_{\ell}$ acts as multiplication by $\text{exp}(\gamma)$ and let $\mathcal{E}_{\chi}^{\text{uni}}$ be the corresponding $\ell$-adic local system on $T$. Since $W_{\chi}^{\circ}$ is normal in $W_{\chi}$, we get a $W_{\chi}$-equivariant structure on $\mathcal{E}_{\chi}^{\text{uni}}$ which gives rise to a $W_{\chi}$-equivariant structure on $\mathcal{E}_{\chi} = \mathcal{E}_{\chi}^{\text{uni}} \otimes \mathcal{L}_{\chi}$. Finally, define $\mathcal{E}_{\Theta} = \text{Ind}_{\W_{\chi}}^{W} \mathcal{E}_{\chi}$ ($\Theta$ being the $W$-orbit of $\chi$). The local system $\mathcal{E}_{\Theta}$ has rank $|W||W_{\chi}|/|W_{\chi}^{\circ}|$.
We now recall some examples in \cite{chen1} and make some additional observations:
\begin{altenumerate}
\item If $\chi$ is regular, i.e. $W_{\chi} = e$, then we have $\mathcal{E}_{\chi} = \mathcal{L}_{\chi}$  and $\mathcal{E}_{\Theta}=\text{Ind}_e^{W}\mathcal{L}_{\chi} \cong \oplus_{w\in W} \mathcal{L}_{w\chi}$, hence we get the regular representation of $W$ on the fiber at identity, through the $W$-equivariant structure of $\mathcal{E}_{\Theta}$. 
\item If $\chi$ is the trivial character, then $W_{\chi} = W_{\chi}^{\circ} = W$ and is the local system of rank $|W|$ corresponding to the unipotent representation of $R_{\chi} = R_T/\langle R_{T,+}^{W} \rangle$ of $\pi_1^{t}(T)$. Here, we again get the regular representation of $W$. We expand on this point below with a lemma.\\
But first, we make the observation that it is enough for us to \textbf{prove all of our statements for groups with connected centre}: the main point is to construct non-trivial extensions as outlined in Remark \ref{r1} and to this end, we can simply pull back those which we will construct for the quotient group with connected centre.\\
So, in all of the following, we assume that our connected reductive group $G$ also has connected centre.\\
First, from \cite{DL}, Theorem 5.13, we have that after our reduction to the connected centre case, given $\chi \in \mathcal{C}(T), W_{\chi}^{\circ}=W_{\chi}$.
\end{altenumerate}

\begin{lemma}
Given a character $\chi \in \mathcal{C}(T)$ of the torus $T$, we have that for $W_{\chi}^{\circ} = W_{\chi} \subseteq W$ the map $R_T^{W_{\chi}^{\circ}} \to R_T$ is flat. $R_{T,+}\subset R_T$ is the argumentation ideal and the fibre over the maximal ideal  $R_{T,+}^{W_{\chi}^{\circ}} \subset R_T^{W_{\chi}^{\circ}}$ for the corresponding finite map between schemes is  $R_T/\langle R_{T,+}^{W_{\chi}^{\circ}} \rangle$. This finite dimensional vector space is isomorphic to the regular representation of $W_{\chi}^{\circ}$ under the natural $W_{\chi}^{\circ}$ action.    
\end{lemma}
\begin{proof}
As we have identifications $\pi_1(T)_{\ell} \cong X_{\ast}(T)\otimes_{\Z}\Z_{\ell}(1)$, it is equivalent to look at the action of the Weyl group $W_{\chi}^{\circ}$ on the affine space corresponding to $X_{\ast}(T)$ over $\overline{\Q_{\ell}}$. As the Weyl group is generated by reflections, the quotient space is smooth and via miracle flatness (\cite{mfl}, \href{https://stacks.math.columbia.edu/tag/00R4}{Lemma 00R4}), we can say that the quotient map is flat: both the spaces involved are smooth and the map is finite. We are interested in the $W_{\chi}^{\circ}$ representation of the fibre over $0$, knowing that generically it is the $W_{\chi}^{\circ}$-regular representation (as the action is free).

As a finite flat module is locally free, we have that that $R_T$ is a projective module over $R_T^{W_{\chi}^{\circ}}$. So, we get that the finite group action can be restricted to Zariski open sets in $\text{Spec}(R_T^{W_{\chi}^{\circ}})$ (those over which we get free modules) and studied, giving us various representations of free modules of rank $|W_{\chi}^{\circ}|$. This setup enforces that the representations do not vary because of the fact that characters determine representations uniquely: say over a Zariski open Spec $S\subset \text{Spec}(R_T^{W_{\chi}^{\circ}})$, we consider the induced map $W_{\chi}^{\circ} \to GL(|W_{\chi}^{\circ}|,S)$. The various elements in $S$ (obtained through character relations) are constants in an open set in Spec $S$ (because, generically, the representation is regular due to the free action of $W_{\chi}^{\circ}$) so they are the same constants everywhere in Spec $S$. Therefore, we get the regular representation of $W_{\chi}^{\circ}$ at the fiber over $0$ as well.
\end{proof}

We have, from \cite{chen1}, Corollary 5.2 that $\mathcal{E}_{\Theta}$ is $\ast$-central in the sense of \cite{chen1}, meaning that the action of $W_{\chi}^{\circ}$ on the cohomology groups on $H^{\ast}(T,\mathcal{E}_{\Theta})$ is through the sign representation, so consequently, the action of $W_{\chi}^{\circ}$ on $H^{\ast}_c(T, \mathcal{E}_{\Theta}^{\vee} \otimes \text{sgn})$ is trivial (due to Verdier duality), so $\mathcal{E}_{\Theta}^{\vee} \otimes \text{sgn}$ is central in the sense of \cite{bd}. This allows us to say that $\mathcal{E}_{\Theta}^{\vee} \otimes \text{sgn}$ is an object in $\mathscrsfs{D}_W^{\circ}(T)$.

\begin{prop}\label{ftheta}
Given $\chi \in \mathcal{C}(T)$, the corresponding object $\mathcal{F}_{\Theta}:= \mathcal{E}_{\Theta}^{\vee}\otimes \text{sgn}$ admits a non-trivial morphism $\mathcal{F}_{\Theta} \to \delta_1$ in the category $\mathscrsfs{D}_W^{\circ}(T)$ with the $W$ action on $\delta_1$ being the trivial action.
\end{prop}

\begin{proof}
From the lemma above, it is clear that $\mathcal{E}_{\chi}:= \mathcal{E}_{\chi}^{\text{unip}} \otimes \mathcal{L}_{\chi}$ has the regular representation of $\mathcal{W_{\chi}}$ at the fibre at identity, and so $\mathcal{E}_{\Theta}:= \text{Ind}_{W_{\chi}}^{W} \mathcal{E}_{\chi}$ gets the regular representation of $W$ at identity after the $W$-induction of the regular representation of $W_{\chi}^{\circ}$. Therefore, even the dual representation is regular and gets the sign representation as a quotient.\\
This allows us to have a map $\mathcal{F}_{\Theta} \to \delta_1$ through the adjunction map for the inclusion of identity point $e \to T$ that is equivariant for the $W$-action, with the trivial action on $\delta_1$.\\
\end{proof}

Consider the quotient map $\pi_{\chi}: \mathcal{C}(T) \to \mathcal{C}(T)// W_{\chi}$. Let $0 \in \mathcal{C}(T)(\overline{\Q_{\ell}})$ be the
trivial character and let $D_{\chi} = \pi_{\chi}^{-1}(\pi_{\chi}(0))$ be the scheme theoretic preimage of $\pi_{\chi}(0) \in \mathcal{C}(T )// W_{\chi}$ for the map $\pi_{\chi}$. We have the following coherent sheaves on $\mathcal{C}(T)$: $\mathcal{R}_{\chi}^{uni} = \mathcal{O}_{D_{\chi}}, \mathcal{R}_{\chi} = m_{\chi}^{\ast}(\mathcal{O}_{D_{\chi}}), \mathcal{R}_{\Theta} = \text{Ind}_{W_{\chi}}^W\mathcal{R}_{\chi}$ (Where $m_{\chi}: \mathcal{C}(T) \to \mathcal{C}(T)$ is the morphism for translation by $\chi$).\\
We have, from Proposition 5.1 of \cite{chen1}, that there is an isomorphism $\mathcal{M}_{\ast}(\mathcal{E}_{\Theta}\otimes \text{sign}_W) \cong \mathcal{R}_{\Theta}$. Now, from the equivalent property (iii) of \ref{equivprop} (we have $W_{\chi} = W_{\chi}^{\circ}$ under the connected centre assumption) and the fact that the Mellin transform restricts to an equivalence for monodromic objects, as $\mathcal{R}_{\chi}^{uni}$ is the pull back of a simple object from the GIT quotient, we see that in the $W-$equivariant category $\text{Perv}_W(T),$ $ \mathcal{E}_{\Theta}[\text{dim}T]$ is a simple object and hence, so is $\mathcal{F}_{\Theta}[\text{dim} T]$ in $\text{Perv}_W^{\circ}(T).$  

We now recall corollary 1.7 of \cite{bd},
there is an equivalence of abelian categories:
$$\text{Perv}^{\circ}_W(T) \cong \text{Perv}^{\circ}_G(G) \cong \text{Perv}(e_{\mathcal{L}}\mathscrsfs{D}(G)e_{\mathcal{L}})[\text{dim } U]$$
We refer to loc. cit. for detailed description of the functors (but briefly, the first one is the $W$-equivariant parablolic induction and the second one is given by convolution by $e_{\mathcal{L}}$ and the composition is the restriction of $\zeta^{-1}$ talked about at the start of this section).\\
We now prove a key result, an analogue of \ref{const} for the category $e_{\mathcal{L}}\mathscrsfs{D}(G)e_{\mathcal{L}}$ and therefore also for $\mathscrsfs{D}_W^{\circ}(T)$:
\begin{prop}\label{endid}
All the natural transformations between the identity functors on $e_{\mathcal{L}}\mathscrsfs{D}(G)e_{\mathcal{L}}$ which are induced from the Perverse subcategory $\text{Perv}(e_{\mathcal{L}}\mathscrsfs{D}(G)e_{\mathcal{L}})$ are given by multiplication by a scalar $c\in \overline{\Q_{\ell}}$ and hence the same also holds true for $\mathscrsfs{D}_W^{\circ}(T)$.    
\end{prop}

\begin{proof}
The equivalence referred to before this proposition gives us $\text{Perv}^{\circ}_W(T) \cong \text{Perv}(e_{\mathcal{L}}\mathscrsfs{D}(G)e_{\mathcal{L}})[\text{dim } U] $ which induces the equivalence of the respective derived categories. So in particular, there is a correspondence between endomorphisms of identity functors induced from the respective perverse subcategories. So in particular, we are in the setup of \ref{r1}, \ref{constp}, \ref{conarg}.

We have that the object $\delta_1$ with trivial $W$-equivariance lies in $\mathscrsfs{D}_W^{\circ}(T)$. A given natural transformation (between the identity functors) will act by multiplication by a scalar on it, lets fix it as $c\in \overline{\Q_{\ell}}$. 

Recall the equivalence set up in \ref{dictionary} between $e_{\mathcal{L}}\mathscrsfs{D}(BwB)e_{\mathcal{L}}$ and $\mathscrsfs{D}(N_{w,\mathcal{L}})$. Let $C_{w,\beta} \in e_{\mathcal{L}}\mathscrsfs{D}(G)e_{\mathcal{L}}$ be the sheaf given under this equivalence by the constant sheaf $\overline{\Q_{\ell}}$ on the connected component $N_{\beta, w, \mathcal{L}} \subseteq N_{w,\mathcal{L}}$, indexed by $\beta \in \pi_0({N_{w,\mathcal{L}}})$. Let us denote the corresponding object under $\zeta$ by $\mathcal{P} \in \mathscrsfs{D}_W^{\circ}(T)$ (refer \ref{zetathm}).\\

From the Mellin transform set up in \cite{GL}, for the non-zero object $\mathcal{P}$, there must be a $\chi\in \mathcal{C}(T)$ such that $\mathcal{P}\otimes \mathcal{L}_{\chi} \neq 0$, as from loc. cit., only the zero object goes to the zero object.
We now consider $\mathcal{F}_{\Theta}$ for $\Theta$ being the $W$- orbit of $\chi^{-1}$ in $\mathcal{C}(T)$ so that the point corresponding to $\chi$ on $\mathcal{C}(T)$ is one of the finitely many points that make up the support of $\mathcal{F}_{\Theta}$. We have from \ref{ftheta}, a map $\mathcal{F}_{\Theta} \to \delta_1$ in the category $\mathscrsfs{D}_W^{\circ}(T)$, along with the fact that $\mathcal{F}_{\Theta}[\text{dim }T]$ is a simple object in $\text{Perv}_W^{\circ}(T)$ and therefore, the endomorphisms of $\mathcal{F}_{\Theta}$ are given by $\overline{\Q_{\ell}}$.\\

Convolving with $\mathcal{P}$, we get $\mathcal{P}\ast \mathcal{F}_{\Theta} \to \mathcal{P}$ a non-trivial map from a monodromic sheaf to $\mathcal{P}$ (non-trivial because of the choice of $\chi$). The analogous version of Proposition 5.3 in \cite{chen1} for $\mathcal{F}_{\Theta}$, tells us that there is a non-trivial map from $\mathcal{F}_{\Theta}$ to $\mathcal{P}$ (upto a shift) in $\mathscrsfs{D}_W^{\circ}(T)$.\\

Now, passing back to $e_{\mathcal{L}}\mathscrsfs{D}(G)e_{\mathcal{L}}$ through $\zeta^{-1}$, let's denote by $\mathcal{Q} \in e_{\mathcal{L}}\mathscrsfs{D}(G)e_{\mathcal{L}}$ the image of $\mathcal{F}_{\Theta}$. We therefore have a map from $\mathcal{Q}$ to $C_{w,\beta}$ upto a shift (as $\mathcal{P}$ was defined to be the image of $C_{w,\beta}$ under $\zeta$). Secondly, we have that the endomorphisms of $\mathcal{Q}\in e_{\mathcal{L}}\mathscrsfs{D}(G)e_{\mathcal{L}}$ are given by $\Qlcl$, because that is true for $\mathcal{F}_{\Theta}$.\\

For the given natural transformation, the condition on $\mathcal{Q}$ makes all the scalars $c_{v,\gamma}$ (see \ref{scalardefn}) for those pairs $(v,\gamma)$ for which $X_{v,\gamma}$ lies in the support of $\mathcal{Q}$, to be equal: If we know the endomorphisms of a certain sheaf to be only scalars $c\in \overline{\Q_{\ell}}$, then for a given natural transformation, that scalar can be determined through restricting the sheaf on suitable strata and knowing the scalar there (Multiplication by $c_{v,\gamma}$ denotes the natural transformation 
on sheaves supported on $UX_{v,\gamma}U$).

The image of $\delta_1$ under $\zeta^{-1}$ is $e_{\mathcal{L}}[\text{dim }U]$. As the support of $e_{\mathcal{L}}$ is $U$ which contains the identity of the group $G$, we have that $c_{id,\beta}=c$. Secondly, there is a non-trivial map $\mathcal{F}_{\Theta} \to \delta_1$ which under $\zeta^{-1}$, gives a non-trivial map $\mathcal{Q}\to e_{\mathcal{L}}[\text{dim } U]$ in $e_{\mathcal{L}}\mathscrsfs{D}(G)e_{\mathcal{L}}$. Therefore, the scalars associated to $\mathcal{Q}$ are all equal to $c$.

The existence of the non-trivial morphism from $\mathcal{Q}$ to $C_{w,\beta}$ (upto a shift), shows the existence of non-trivial extension as talked about in Corollary \ref{constp}, Remark \ref{conarg} which in turn shows that $c_{w,\beta} = c$ (this is a consequence of the gluing setup: a non-trivial extension of sheaves on distinct strata equates the two associated scalars).

As the pair $(w,\beta)$ was chosen arbitrarily, we have that $c_{w,\beta}=c$ for all pairs $(w,\beta)$, and hence, the natural transformation on the whole of $e_{\mathcal{L}}\mathscrsfs{D}(G)e_{\mathcal{L}}$, is expressed by multiplication by the scalar $c$ (as discussed in \ref{constp}, they are determined by their restrictions to Bruhat cells).
\end{proof} 
\begin{prop}
Forgetting the $W$-equivariant structure of $\mathscrsfs{D}_W^{\circ}(T)$, the functor 
$\zeta \circ \bm{\Psi}^{\ast}: e_{\mathcal{L}}\mathscrsfs{D}(G)e_{\mathcal{L}} \to \mathscrsfs{D}(T)$ admits a natural isomorphism to the functor $^{w_0}(e_{U^{-}}\circ {\textunderscore})_{|T}: e_{\mathcal{L}}\mathscrsfs{D}(G)e_{\mathcal{L}} \to \mathscrsfs{D}(T)$. (Here we have abused notation and taken $\zeta$ to be a functor to $\mathscrsfs{D}(T)$, composing with the forgetful functor $\mathscrsfs{D}_W^{\circ}(T) \to \mathscrsfs{D}(T))$. \end{prop}

\begin{proof}
The proposition above claims that there is a natural transformation $\pi$ making the following diagram commute:  
\[\xymatrix{
 e_\mathcal{L}\mathscrsfs{D}(G) e_{\mathcal{L}}\ar[r]^{\bm{\Psi}^{\ast}} \ar[d]^{\zeta} & e_{\mathcal{L}}\mathscrsfs{D}(G)e_{\mathcal{L}} \ar[d]^{\zeta}\\
 \mathscrsfs{D}(T) \ar[r]^{\bm{\Psi}_T^{\ast}} & \mathscrsfs{D}(T)  .
}\] 

This is because in \ref{psiprop}, the homomorphism $\bm{\tau}$ was constructed using the fact that a non-degenerate local system determines pinning on $G$ uniquely along with the underlying diagram homomorphism being $\alpha \to -w_0(\alpha)$. As $\bm{\Psi}:= \bm{\tau} \circ \iota = \iota \circ \bm{\tau}$, where $\iota$ is the inverse map, we see that on the chosen maximal torus $T$, $\bm{\Psi}$ acts exactly by $w_0$ conjugation. Let's denote by $\bm{\Psi}_{T}$ the restriction of $\bm{\Psi}$ to $T$. 

We have that for $A,B \in \mathscrsfs{D}(G), \bm{\Psi}^{\ast}(A \ast B) = (\bm{\Psi}^{\ast}B \ast \bm{\Psi}^{\ast}A)$ because $\bm{\Psi}^{\ast}$ is an anti-involution. Restricting to $T$ we get $\bm{\Psi}_{T}^{\ast}(A \ast B) = (\bm{\Psi}^{\ast}B \ast \bm{\Psi}^{\ast}A)_{|T}$.

Recall, from the start of this section that forgetting the $W$-equivariant structure, the functor $\zeta$ above is described as $A\in e_{\mathcal{L}}\mathscrsfs{D}(G)e_{\mathcal{L}} \mapsto (e_{U^{-}} \ast A)_{|T}[-4 \text{dim} U] (-2 \text{dim} U)$.
As the composition of shift and twist functors (more precisely, $[-4 \text{dim} U] (-2 \text{dim} U)$) appears on both the sides, it is enough for us to describe a natural transformation between the functors before applying these operations.

Now, the relevant functor for the top-right composition is:\\ $A\in e_{\mathcal{L}}\mathscrsfs{D}(G)e_{\mathcal{L}} \mapsto (e_{U^{-}} \ast \bm{\Psi}^{\ast}A)_{|T} \cong \bm{\Psi}^{\ast}_{T}(A \ast \bm{\Psi}^{\ast}e_{U^{-}})_{|T} \cong \bm{\Psi}^{\ast}_{T}(A \ast e_{U^{-}})_{|T}\cong \bm{\Psi}^{\ast}_{T}(e_{U^{-}}\ast A)_{|T}.$ 

We get the first isomorphism through the anti-monoidal and involutive property of $\bm{\Psi}^{\ast}$ along with the fact above. The second isomorphism is due to the fact that $\bm{\Psi}(U^{-}) = U^{-}.$ Finally, since these restrictions to $T$ are determined by the restriction of $A$ to $B^{-}$ (because we are convolving with a sheaf supported on $U^{-}$ followed by restriction to $T$) we have, because of the normality of $U^{-}$ in $B^{-}$ that, we can switch the order of convolution with $e_{U^{-}}$, giving us the third isomorphism.

So we have our result as $\bm{\Psi}_T$ is just conjugation by $w_0$, taking $\pi$ (the required natural transformation) to be the composition of the inverses of the isomorphisms described above.
\end{proof}
 
With the natural transformation set up between the two functors above, we now have the following that sets up the dictionary between the functors and natural transformations discussed for the two categories $e_{\mathcal{L}}\mathscrsfs{D}(G)e_{\mathcal{L}}$ and $\mathscrsfs{D}_W^{\circ}(T)$.

\begin{prop}\label{zetasq}

There is a natural isomorphism making the following diagram commute: 

\[\xymatrix{
 e_\mathcal{L}\mathscrsfs{D}(G) e_{\mathcal{L}}\ar[r]^{\bm{\Psi}^{\ast}} \ar[d]^{\zeta} & e_{\mathcal{L}}\mathscrsfs{D}(G)e_{\mathcal{L}} \ar[d]^{\zeta}\\
 \mathscrsfs{D}_W^{\circ}(T) \ar[r]^{s_{w_0}} & \mathscrsfs{D}_W^{\circ}(T)  .
}\] 

And secondly, the natural transformation $\theta: \bm{\Psi}^{\ast} \to \text{Id}$ goes to the natural transformation $\alpha_{w_0}: s_{w_0}\to \text{Id}$, under the equivalence $\zeta$ ($s_{w_0}$ and $\alpha_{w_0}$ are as defined in \ref{natmaps}).    
\end{prop}
\begin{proof}
The image of the map $\theta: \bm{\Psi}^{\ast}\mathcal{F} \to \mathcal{F}$ (given by the natural transformation) under $\zeta$, after appropriately shifting and twisting gives us the map:  $(e_{U^{-}}\ast\bm{\Psi}^{\ast}\mathcal{F})_{|T} \to (e_{U^{-}}\ast \mathcal{F})_{|T}$. This map is in fact $W$-equivariant, respecting the $W$-equivariant structures which these two sheaves get through the action of the Kazhdan-Laumon sheaves described at the start of this section (refer \cite{bd} for more details).\\
However, the previous proposition allows us to impart $^{w_0}(e_{U^{-}}\ast \mathcal{F})_{|T}$ with a $W$-equivariant structure, borrowing the one from $(e_{U^{-}}\ast\bm{\Psi}^{\ast}\mathcal{F})_{|T}$, using $\pi$ as described.\\
On the other hand, in \ref{natmaps}, we have defined some natural transformations $\alpha_{w}$ on the category $\mathscrsfs{D}_W^{\circ}(T)$, in particular, we have that the natural transformation $\alpha_{w_0}$ will induce the map $(e_{U^{-}}\ast\mathcal{F})_{|T} \to {^{w_0}}(e_{U^{-}}\ast\mathcal{F})_{|T}$ coming from the action of the Kazhdan-Laumon sheaf $\mathcal{K}_{w_0}$.
So we get the following sequence:\\
\[^{w_0}(e_{U^{-}}\ast\mathcal{F})_{|T} \xrightarrow[]{\pi} (e_{U^{-}}\ast\bm{\Psi}^{\ast}\mathcal{F})_{|T} \xrightarrow[]{\theta} (e_{U^{-}}\ast \mathcal{F})_{|T} \xrightarrow[]{\alpha_{w_0}} {^{w_0}}(e_{U^{-}}\ast \mathcal{F})_{|T}\]
\\
All the maps are $W$-equivariant and so, after conjugating the composition of the above by $w_0$, we obtain a natural transformation of the identity functor on $\mathscrsfs{D}_W^{\circ}(T).$\\
Now, appealing to the result in proposition \ref{endid}, we have that the composition is given by a scalar $c$, which will be the one describing the endomorphism of the object $\delta_1$ with trivial $W$-equivariance in $\mathscrsfs{D}_W^{\circ}(T).$\\
However, since $\delta_1$, under $\zeta^{-1}$ gives the object $e_{\mathcal{L}}[\text{dim }U],$ and because $\bm{\Psi}^{\ast}e_{\mathcal{L}} \to e_{\mathcal{L}}$ is the canonical morphism, it is easy to see that the above composition for $\delta_1$ is in fact given by multiplication by $1$.\\
This, in turn, enforces that the composition of the sequence above is in fact identity.
It therefore follows that $s_{w_0}$ is the functor corresponding to $\bm{\Psi}^{\ast}$ and $\alpha_{w_0}$ is the image of $\theta$ (due to the fact $\alpha_{w_0} \circ \alpha_{w_0} = \text {Id})$.   
\end{proof}

\bibliographystyle{alpha}
\bibliography{main}
\end{document}